\documentclass[a4paper,12pt]{amsart}
\usepackage{amsfonts}
\usepackage{amssymb}
\usepackage{ifthen}
\usepackage{graphicx}
\usepackage{enumitem}
\usepackage{amsmath}
\usepackage{marvosym}
\usepackage{graphicx}
\usepackage{lipsum}

\nonstopmode \numberwithin{equation}{section}
\usepackage[usenames]{color}
\newtheorem{thm}{Theorem}[section]
\newtheorem{lem}{Lemma}[section]
\newtheorem{cor}[thm]{Corollary}
\newtheorem{prop}[thm]{Proposition}

\theoremstyle{definition}
\newtheorem{mlem}{Main lemma}[section]
\newtheorem{assertion}{Assertion}[section]
\newtheorem{cl}{Claim}[section]
\newtheorem{ca}{Case}[section]
\newtheorem{sca}{Subcase}[section]
\newtheorem{scl}{Subclaim}[section]
\newtheorem{conj}[thm]{Conjecture}
\newtheorem{fact}{Fact}[section]
\newtheorem{defn}{Definition}[section]
\newtheorem{prob}{Problem}[section]
\newtheorem{ques}[thm]{Question}
\newtheorem{rem}{Remark}[section]
\newtheorem{exam}{Example}[section]

\numberwithin{equation}{section}
\DeclareMathOperator*{\esssup}{ess\,sup}

\newcounter {own}
\def\theown {\thesection       .\arabic{own}}

\newlist{steps}{enumerate}{1}
\setlist[steps,1]{
  leftmargin=*,
  label=\textbf{Step \arabic*}.,
  ref=Step~\arabic*,
}

\newenvironment{pf}[1][]{%
 \vskip 3mm
 \noindent
 \ifthenelse{\equal{#1}{}}%
  {{\slshape Proof. }}%
  {{\slshape #1.} }%
 }%
{\qed\bigskip}

\newcounter{alphabet}
\newcounter{tmp}
\newenvironment{Thm}[1][]{\refstepcounter{alphabet}%
\bigskip%
\noindent%
{\bf Theorem \Alph{alphabet}}%
\ifthenelse{\equal{#1}{}}{}{ (#1)}%
{\bf .} \itshape}{\vskip 8pt}

\newcounter{alphabet2}

\makeatletter
\newcommand{\Ref}[1]{\@ifundefined{r@#1}{}{\setcounter{tmp}{\ref{#1}}\Alph{tmp}}}
\makeatother

\newcommand{\ID}{{\mathbb D}}

\def\be{\begin{equation}}
\def\ee{\end{equation}}

\newcommand{\ben}{\begin{enumerate}}
\newcommand{\een}{\end{enumerate}}

\newcommand{\blem}{\begin{lem}}
\newcommand{\elem}{\end{lem}}
\newcommand{\bthm}{\begin{thm}}
\newcommand{\ethm}{\end{thm}}
\newcommand{\bcor}{\begin{cor}}
\newcommand{\ecor}{\end{cor}}
\newcommand{\beg}{\begin{exam}}
\newcommand{\eeg}{\end{exam}}
\newcommand{\begs}{\begin{examples}}
\newcommand{\eegs}{\end{examples}}
\newcommand{\bdefe}{\begin{defn}}
\newcommand{\edefe}{\end{defn}}
\newcommand{\bprob}{\begin{prob}}
\newcommand{\eprob}{\end{prob}}
\newcommand{\bques}{\begin{ques}}
\newcommand{\eques}{\end{ques}}
\newcommand{\bei}{\begin{itemize}}
\newcommand{\eei}{\end{itemize}}
\newcommand{\bcon}{\begin{conj}}
\newcommand{\econ}{\end{conj}}
\newcommand{\bop}{\begin{op}}
\newcommand{\eop}{\end{op}}

\newcommand{\bas}{\begin{assertion}}
\newcommand{\eas}{\end{assertion}}

\newcommand{\bfa}{\begin{fact}}
\newcommand{\efa}{\end{fact}}

\newcommand{\bca}{\begin{ca}}
\newcommand{\eca}{\end{ca}}
\newcommand{\bsca}{\begin{sca}}
\newcommand{\esca}{\end{sca}}

\newcommand{\bcl}{\begin{cl}}
\newcommand{\ecl}{\end{cl}}

\newcommand{\bmlem}{\begin{mlem}}
\newcommand{\emlem}{\end{mlem}}

\newcommand{\bscl}{\begin{scl}}
\newcommand{\escl}{\end{scl}}

\newcommand{\bcons}{\begin{conjs}}
\newcommand{\econs}{\end{conjs}}

\newcommand{\bprop}{\begin{prop}}
\newcommand{\eprop}{\end{prop}}

\newcommand{\br}{\begin{rem}}
\newcommand{\er}{\end{rem}}
\newcommand{\brs}{\begin{rems}}
\newcommand{\ers}{\end{rems}}
\newcommand{\bo}{\begin{obser}}
\newcommand{\eo}{\end{obser}}
\newcommand{\bos}{\begin{obsers}}
\newcommand{\eos}{\end{obsers}}
\newcommand{\bpf}{\begin{pf}}
\newcommand{\epf}{\end{pf}}
\newcommand{\ba}{\begin{array}}
\newcommand{\ea}{\end{array}}
\newcommand{\beq}{\begin{eqnarray}}
\newcommand{\beqq}{\begin{eqnarray*}}
\newcommand{\eeq}{\end{eqnarray}}
\newcommand{\eeqq}{\end{eqnarray*}}

\newcounter{minutes}
\divide\time by 60
\newcounter{hours}
\multiply\time by 60 \addtocounter{minutes}{-\time}

\begin{document}
\newcommand\blfootnote[1]{%
\begingroup
\renewcommand\thefootnote{}\footnote{#1}%
\addtocounter{footnote}{-1}%
\endgroup
}

\bibliographystyle{amsplain}

\title{$L^p$-theory for Cauchy-transform on the unit disk}
\author{David Kalaj}
\address{David Kalaj, University of Montenegro, Faculty of Natural Sciences and
Mathematics, Cetinjski put b.b. 81000 Podgorica, Montenegro}
\email{davidk@ucg.ac.me}

\author{Petar Melentijevi\'c}
\address{Petar Melentijevi\'c, Matemati\v cki fakultet, University of Belgrade, Serbia}
\email{petarmel@matf.bg.ac.rs}

\author{Jian-Feng Zhu}
\address{Jian-Feng Zhu, School of Mathematical Sciences,
Huaqiao University,
Quanzhou 362021, People's Republic of China and Department of Mathematics,
Shantou University, Shantou, Guangdong 515063, People's Republic of China.
} \email{flandy@hqu.edu.cn}

\date{}

\subjclass[2000]{Primary 42B20, 42B38}
\keywords{Cauchy transform, Beurling transform, Bergman projection, Hilbert norm, Bessel function.\\
}

\begin{abstract}
Let $\mathbb{D}$ be the unit disk and $\varphi\in L^p(\mathbb{D}, \mathrm{d}A)$, where $1\leq p\leq\infty$. For $z\in\mathbb{D}$, the Cauchy-transform on $\mathbb{D}$, denote by $\mathcal{P}$, is defined as follows:
$$\mathcal{P}[\varphi](z)=-\int_{\mathbb{D}}\left(\frac{\varphi(w)}{w-z}+\frac{z\overline{\varphi(w)}}{1-\bar{w}z}\right)\mathrm{d}A(w).$$
The Beurling transform on $\mathbb{D}$, denote by $\mathcal{H}$, is now defined as the $z$-derivative of $\mathcal{P}$.
In this paper, by using Hardy's type inequalities and Bessel functions, we show that $\|\mathcal{P}\|_{L^2\to L^2}=\alpha\approx1.086$, where
$\alpha$ is a solution to the equation: $2J_0(2/\alpha)-\alpha J_1(2/\alpha)=0$, and $J_0$, $J_1$ are Bessel functions.
Moreover, for $p>2$, by using Taylor expansion, Parseval's formula and hypergeometric functions, we also prove that $\|\mathcal{P}\|_{L^p\to L^{\infty}}=2(\Gamma(2-q)/\Gamma^2(2-\frac{q}{2}))^{1/q}$, where $q=p/(p-1)$ is the conjugate exponent of $p$, and $\Gamma$ is the Gamma function.
Finally, applying the same techniques developed in this paper, we show that the Beurling transform $\mathcal{H}$ acts as an isometry of $L^2(\mathbb{D}, \mathrm{d}A)$.
\end{abstract}
\thanks{}

\maketitle \pagestyle{myheadings} \markboth{David Kalaj, Petar Melentijevi\'c, and Jian-Feng Zhu}{$L^p$-theory for Cauchy-transform on the unit disk}


\section{Introduction}\label{sec-1}
Let $\ID=\{z:|z|<1\}$ be the unit disk of $\mathbb{C}$ and $\mathbb{T}=\{\zeta: |\zeta|=1\}$ the unit circle. Denote by $L^p(\ID, \mathrm{d}A)$ $(1\leq p\leq\infty)$ the space of
complex-valued measurable functions on $\ID$ with finite integral
$$\|\varphi\|_{p}=\left(\int_{\ID}|\varphi(z)|^p\mathrm{d} A(z)\right)^{\frac{1}{p}},\ \ \ 1\leq p<\infty,$$
where
$$\mathrm{d}A(z) =\frac{1}{\pi}\mathrm{d}x\mathrm{d}y=\frac{1}{\pi}r\mathrm{d}r\mathrm{d}\theta,\ \ \ z=x+iy=re^{i\theta},$$
is the normalized area measure on $\ID$ (cf. \cite[Page 1]{Hendenmalm}). For the case $p=\infty$, we let $L^{\infty}(\ID, \mathrm{d}A)$ denote the space of (essentially) bounded
functions on $\ID$. For $\varphi\in L^{\infty}(\ID, \mathrm{d}A)$, we define
$$\|\varphi\|_\infty=\esssup\{|\varphi(z)|: z\in\ID\}.$$
The space $L^{\infty}(\ID, \mathrm{d}A)$ is a Banach space with the above norm  (cf. \cite[Page 2]{Hendenmalm}).

Recall that the norm of an operator $T: X\rightarrow Y$ between two normed spaces $X$ and $Y$ is defined by
$$\|T\|_{X\rightarrow Y}=\sup\{\|Tx\|_Y:\|x\|_X=1\}.$$
For simplicity, in this paper, if $X=L^{p}(\ID, \mathrm{d}A)$ and $Y=L^{q}(\ID, \mathrm{d}A)$, then we write $\|T\|_{L^p\to L^q}$ instead of $\|T\|_{L^p(\ID,\, \mathrm{d}A)\to L^q(\ID,\, \mathrm{d}A)}$ for the norm of the operator $T$. Moreover, for the case of $X=Y=L^{p}(\ID, \mathrm{d} A)$, we then write $\|T\|_p$ instead of $\|T\|_{L^{p}\rightarrow L^{p}}$ for the $L^p$ norm of $T$.

The inhomogeneous Cauchy-Riemann system
\be\label{622-cau-eq}\frac{\partial f}{\partial \bar{z}}=\varphi\ee
can be solved uniquely if one can fix the domain of $z$ and impose some boundary conditions on $f$.

Suppose $z\in\ID$ and $f\in W^{1,\, p}(\ID)$, the Sobolev space of $\ID$, i.e., the partial derivatives of $f$ is of $L^p(\ID, \mathrm{d}A)$ space.
If the real part of $f$  vanishes on $\mathbb{T}$, then the solutions to (\ref{622-cau-eq}) can be given explicitly
by the following formula (cf. \cite[Page 151, (4.134)]{Astala}):
\be\label{ope-P}f(z)=\mathcal{P}[\varphi](z)=-\int_{\ID}\left(\frac{\varphi(w)}{w-z}+\frac{z\overline{\varphi(w)}}{1-\bar{w}z}\right)\mathrm{d}A(w).\ee

Here the first integral is known as the {\it Cauchy integral operator} on $\ID$, denote by $\mathfrak{C}$, which is given by (cf.  \cite{ Baranov, Dostanic-plms})
$$\mathfrak{C}[\varphi](z)=\int_{\ID}\frac{\varphi(w)}{w-z}\mathrm{d} A(w).$$
According to \cite[Theorem 2.1]{Anderson2}, we know that $\mathfrak{C}$ is compact from $L^2(\ID, \mathrm{d} A)$ to $L^2(\ID, \mathrm{d} A)$.
Moreover, it was proved in \cite[Theorem 1]{Anderson} that
$$\|\mathfrak{C}\|_{2}=\frac{2}{j_0},$$
where $j_0\approx2.405$ is the smallest positive zero of the {\it Bessel function} $J_0(x)$, which is given by (\ref{Bessel-function}) below.

Let $\mathfrak{J}_0$ be the operator on analytic functions $h$ in $\ID$ defined by (cf. \cite[Page 9]{Baranov})
$$\mathfrak{J}_0[h](z)=\int_0^zh(w)\mathrm{d}w,\ \ \ \ z\in\ID.$$
Then one can extend it to an integration operator on $L^p(\ID, \mathrm{d}A)$ as follows:
$$\mathfrak{J}_0[\varphi](z)=\int_{\ID}\frac{z\varphi(w)}{1-\bar{w}z}\mathrm{d}A(w),$$
where $\varphi\in L^p(\ID, \mathrm{d}A)$ and $p\geq1$. Its adjoint operator $\mathfrak{J}_0^*$ is given as follows (cf. \cite[Page 12]{Baranov} ):
$$\mathfrak{J}_0^*[\varphi](z)=\int_{\ID}\frac{\bar{w}\varphi(w)}{1-\bar{w}z}\mathrm{d}A(w).$$

Following the proof of \cite[Theorem 2.1]{Anderson2}, one can easily obtain that $\mathfrak{J}_0$ is also compact from $L^2(\ID, \mathrm{d}A)$ to $L^2(\ID, \mathrm{d}A)$.
Very recently, the authors of this paper showed in \cite[Theorem 1.5]{KZ-JFA} that
$$\|\mathfrak{J}_0^*\|_2=\frac{1}{\sqrt{2}}.$$

Now, it is easy to see that
$$\mathcal{P}[\varphi]=-\mathfrak{C}[\varphi]-\mathfrak{J}_0[\bar{\varphi}],$$
is a compact operator. The following problem becomes interesting:
\begin{prob}\label{prob1}
What is the norm $\|\mathcal{P}\|_2$?
\end{prob}

It should be noted that $\mathcal{P}$ is not a complex linear operator.
Let $d$ be an integer. Denote by $\mathcal{F}_d$ the space of polynomials in $w$ and $\bar{w}$ spanned by the functions $w^m\bar{w}^n$, where $d=m-n$, and $m$, $n$ are arbitrary non-negative integers.

 Contrary to the proof of \cite[Theorem 1]{Anderson}, for $\varphi_i\in\mathcal{F}_i$ and $\varphi_j\in\mathcal{F}_j$, where $i\neq j$, the functions $\mathcal{P}[\varphi_i]$
and $\mathcal{P}[\varphi_j]$ are not necessarily orthogonal in the Hilbert space $L^2(\ID, \mathrm{d}A)$ (see Lemma \ref{orth1-lemma} below).
This makes the problem much harder.
In this paper, we prove the following results:

\begin{thm}\label{general-2norm}

Let  $\alpha\approx1.086$ be the solution to the following equation:
$$2J_0\left(\frac{2}{\alpha}\right)-\alpha J_1\left(\frac{2}{\alpha}\right)=0,$$
where $J_0$, $J_1$ are Bessel functions given by $(\ref{Bessel-function})$ below. Then $$\|\mathcal{P}\|_{2}=\alpha.$$
\end{thm}

The main idea of proving Theorem \ref{general-2norm} is as follows: For $w\in\ID$, let
$$\varphi(w)=\sum_{m,\, n\geq0}a_{m, n}w^m\bar{w}^n=\sum_{d=-\infty}^{\infty}g_d(w),$$
where $a_{m, n}$ are complex numbers, so that only finitely many of them are nonzero and $d=m-n$. Here $$g_d(w)=\sum_{n\geq\max\{0, -d\}}a_{n+d, n}w^m\bar w^n.$$ Then such functions $\varphi$ are dense in $L^2(\ID, \mathrm{d}A)$ (cf. \cite[Page 180]{Anderson}) and
$$\|\mathcal{P}\|_2^2=\sup_{\|\varphi\|_2=1}\frac{\|\mathcal{P}[\varphi]\|_2^2}{\|\varphi\|_2^2}.$$
To find such a supremum, by using {\it Hardy's type inequalities}, Bessel functions and the {\it H\"older's inequality for integral},
we reduce the problem to finding the maximum value of the following function:
$$\frac{\|\mathcal{P}[g_0+g_2]\|_2^2}{\|g_0\|_2^2+\|g_2\|_2^2}.$$
By using {\it Lagrange multiplier method}, we finally get the exact maximum value.

Moreover, by using {\it Taylor expansion}, {\it Parseval's formula} and {\it hypergeometric functions}, we prove the following theorem.
\begin{thm}\label{infinity-norm}
For $p>2$, we have
$$\|\mathcal{P}\|_{L^p\to L^{\infty}}=2{\bigg(\frac{\Gamma(2-q)}{\Gamma^2(2-\frac{q}{2})}\bigg)}^{\frac{1}{q}},$$
where $q=p/(p-1)$ is the conjugate exponent of $p$, and $\Gamma$ is the Gamma function.

If in particular $p=\infty$, then
$$\|\mathcal{P}\|_{\infty}=\frac{8}{\pi}.$$
\end{thm}

By using the {\it Riesz-Thorin interpolation theorem} in the real case (cf. \cite[Theorem 20]{RST}), which is valid for our operator $\mathcal{P}$, we can give the $L^p$ norm estimate for $\mathcal{P}$ as follows:
\begin{cor}\label{Reiz-Thorin}
Let $\alpha=\|\mathcal{P}\|_2\approx 1.086$. Then
$$\|\mathcal{P}\|_{p}\leq \alpha^{\frac{2}{p}}\left(\frac{8}{\pi}\right)^{1-\frac{2}{p}},\ \ \ \mbox{for}\ \ \ p\geq2.$$
The equality can be attained when $p=2$ and $p=\infty$.
\end{cor}

The operator $\mathcal{P}$ itself has some important applications. For example, suppose $f$ is a continuous function in the Sobolev space $W_{loc}^{1, 2}(\ID)$, $\mu$ is an analytic function of $\ID$ with $|\mu|<1$.
The following equation is called {\it the second Beltrami equation}:
\be\label{beltrami}\overline{f_{\bar{z}}}=\mu f_z.\ee
The solutions to (\ref{beltrami}) are harmonic functions, locally quasiregular, but their distortion as measured by the dilatation quotitent (cf. \cite{HS-86})
$$(|f_z|+|f_{\bar{z}}|)/(|f_z|-|f_{\bar{z}}|)=(1+|\mu|)/(1-|\mu|)$$
maybe unbounded at the boundary.

Hengartner and Schober proved in \cite[Theorem 4.1]{HS-86} (see also \cite[Page 126]{Du-04}) that there is a homeomorphic solution $f$ to (\ref{beltrami}), such that $f$ maps $\ID$ onto $\Omega$, a Jordan domain whose boundary is an analytic curve, and $f$ is normalized by $f(0)=w_0$, a fixed point in $\Omega$, and $f_z(0)>0$.
Note that if $\mu=0$, then $f$ is analytic in $\ID$.
This result can be seen as the generalized {\it Riemann mapping theorem} for solutions to the second Beltrami equation.

To prove the existence of such a mapping theorem, they introduced the following integral operator:
\begin{align*}
  \mathcal{T}[\varphi](z) & =-\int_{\ID}\left(\frac{\varphi(w)}{w-z}+\frac{z\overline{\varphi(w)}}{1-\bar{w}z}-\frac{\varphi(w)}{2w}+\frac{\overline{\varphi(w)}}{2\bar{w}}\right)\mathrm{d}A(w) \\
   & =\mathcal{P}[\varphi]+\int_{\ID}\left(\frac{\varphi(w)}{2w}-\frac{\overline{\varphi(w)}}{2\bar{w}}\right)\mathrm{d}A(w).
\end{align*}
Here $\mathcal{T}$ is a compact operator on $L^{p}(\ID, \mathrm{d}A)$ for $p>2$, and has the properties
$$\mbox{Re}\big(\mathcal{T}[\varphi](e^{i\theta})\big)=0,\ \  \mbox{Im}\big(\mathcal{T}[\varphi](0)\big)=0, \ \ \big(\mathcal{T}[\varphi](z)\big)_{\bar{z}}=\varphi.$$
Following their proof, let $\phi:\ID\to \Omega$ be the univalent analytic function with $\phi(0)=w_0$ and $\phi'(0)>0$. Substitute $f=\phi\circ g$. Then $g$ is to be a univalent mapping of $\ID$ onto itselt satisfying the following nonlinear equation:
$$g_{\bar{z}}=\frac{\overline{\phi'\circ g}}{\phi'\circ g}\,\overline{\mu g_z},$$
and has the normalization: $g(0)=0$ and $g_z(0)>0$. By using Wendland's text \cite[Page 68-73]{wendland} and Schauder's fixed-point theorem, they showed that such $g$ exists. In fact, one can choose $g(z)=ze^{\mathcal{T}[\varphi](z)}$. We refer to \cite[Page 477-478]{HS-86} or \cite[Page 131]{Du-04} for more details.

The {\it Beurling transform}  on $\ID$, denote by $\mathcal{H}$, is defined as the $z$-derivative of $\mathcal{P}$, i.e., (cf. \cite[Page 152, (4.136)]{Astala})
\be\label{operator-h}\mathcal{H}[\varphi](z)=\left(\mathcal{P}[\varphi]\right)_z=-p.v.\int_{\ID}\frac{\varphi(w)}{(w-z)^2}\mathrm{d}A(w)-\int_{\ID}\frac{\overline{\varphi(w)}}{(1-\bar{w}z)^2}\mathrm{d}A(w).\ee

It is worth to point out that in the formula of (\ref{operator-h}), the first integral is introduced as the {\it  Beurling transformation} (or {\it Hilbert transformation}), denoted by $\mathcal{S}$, and is taken as a principal value. The second integral is introduced as the {\it Bergman projection}, denoted by $\mathfrak{B}$. Then
$$
\mathcal{H}[\varphi](z)=\mathcal{S}[\varphi](z)-\mathfrak{B}[\bar{\varphi}](z).
$$

Recall that for $\varphi\in L^p(\mathbb{C}, \mathrm{d}A)$, the {\it  Beurling transformation} $\mathcal{S}$ (in the whole plane $\mathbb{C}$) is
the singular integral defined by
$$\mathcal{S}[\varphi](z)=-p.v.\int_{\mathbb{C}} \frac{\varphi(w)}{(z-w)^2}\mathrm{d}A(w):=-\lim_{\varepsilon\to 0}\int_{\mathbb{C}\setminus \mathbb{D}(z,\epsilon)}\frac{\varphi(w)}{(z-w)^2}\mathrm{d}A(w),$$ where $\mathbb{D}(z,\epsilon)=\{w:|z-w|<\varepsilon\}$.
It is known that if $1<p<\infty$, then $\mathcal{S}$ sends $L^p(\mathbb{C}, \mathrm{d}A)$ to $L^p(\mathbb{C}, \mathrm{d}A)$. Moreover,
$\mathcal{S}$ is an isometry in $L^2(\mathbb{C}, \mathrm{d}A)$ (cf. \cite[Page 89]{Ahl}).
However, $\mathcal{S}$ is not an isometry in $L^2(\ID, \mathrm{d}A)$ (see (\ref{beurling-ope}) below).

For the Beurling transform on $\ID$, i.e., $\mathcal{H}$, we have the following result which at the beginning surprised the authors of this paper.
\begin{thm}\label{thm-hilbert-1}
$\mathcal{H}$ is an isometry in $L^2(\ID, \mathrm{d}A)$. In other words, for every $\varphi\in L^2(\ID,\, \mathrm{d}A)$, we have $\|\mathcal{H}[\varphi]\|_{2}=\|\varphi\|_{2}$.
\end{thm}

After we finished this paper, we found that Theorem \ref{thm-hilbert-1} was already given in \cite[Page 151]{Astala}.
Anyway, in Section \ref{sec-4}, we provide a different proof of Theorem \ref{thm-hilbert-1}.

Now if $\varphi\in L^2(\ID, \mathrm{d}A)$, then
$\tilde \varphi(z)=\left\{
\begin{array}{ll}
\varphi(z), & \hbox{if\ \  $z\in \ID$} \\
\ \ \ \ 0, & \hbox{if\ \  $z\in \mathbb{C}\setminus\ID$}
\end{array}
\right.\in L^2(\mathbb{C}, \mathrm{d}A)$.
Therefore, if $z\in \ID$, then $\mathcal{S}[\varphi](z) = \mathcal{S}[\tilde \varphi](z)$.
Thus $$\|\mathcal{S}[\varphi]\|_{2}\le \|\mathcal{S}[\tilde \varphi]\|_{L^2(\mathbb{C},\, \mathrm{d}A)}=\|\tilde \varphi\|_{L^2(\mathbb{C},\, \mathrm{d}A)}=\| \varphi\|_{2}.$$
On the other hand, if $\varphi$ is an analytic function on $\ID$, then $\mathfrak{B}[\bar{\varphi}](z)=\overline{\varphi(0)}$. Therefore, we have
the following corollary.
\begin{cor}
The norm of  the Beurling transformation $\mathcal{S}:L^2(\ID,\, \mathrm{d}A)\to L^2(\ID,\, \mathrm{d}A)$ is $1$ and, if $\varphi$ is an analytic function on $\ID$ such that $\varphi(0)=0$, then we have the following relation $$\|\mathcal{S}[\varphi]\|_{2}=\|\varphi\|_{2}.$$
\end{cor}

The rest of this paper is organized as follows: In Section \ref{sec-2}, we prove six lemmas which are related to the operator $\mathcal{P}$. In Section \ref{sec-3}, we give the proofs of Theorem \ref{general-2norm} and Theorem \ref{infinity-norm}. In Section \ref{L1norm}, we give some calculations related to $\|\mathcal{P}\|_1$.
Section \ref{sec-4} is devoted to the proof of Theorem \ref{thm-hilbert-1}.
\section{Preliminaries}\label{sec-2}
In this section, we should introduce some necessary terminology, and prove six lemmas which will be used in proving our main results.

We start with the definitions of Bessel functions and hypergeometric functions.

\begin{defn}(\cite[Page 15]{watson}) Assume that $\alpha$ is a real number. The {\it Bessel function} $J_\alpha(x)$ is defined as follows:
\be\label{Bessel-function}J_\alpha(x)=\sum\limits_{k=0}^{\infty}\frac{(-1)^k}{\Gamma(k+\alpha+1)!k!}\left(\frac{x}{2}\right)^{2k+\alpha}.\ee
\end{defn}

\begin{defn}(\cite[Page 58]{watson}) The {\it Bessel function of second kind (Hankel's function)}, denote by $\mathcal{N}_\alpha(x)$, is defined as follows:
\be\label{Bessel-function2}\mathcal{N}_\alpha(x)=\frac{J_\alpha(x)\cos \alpha \pi - J_{-\alpha}(x)}{\sin \alpha \pi }.\ee
If $n$ is an integer, then we define $\mathcal{N}_n(x) = \lim_{\alpha\to n}\mathcal{N}_\alpha(x)$.
\end{defn}
\begin{defn}\label{defn-2.1}(\cite[(2.1.2)]{Landrews})
The {\it hypergeometric function}, denote by
$$_pF_q(a_1,a_2,\dots,a_p;b_1,b_2,\ldots,b_q;x),$$
is defined by the following series
\be\label{bz-defn-2.1}_pF_q(a_1,a_2,\dots,a_p;b_1,b_2,\ldots,b_q;x)=\sum_{n=0}^\infty\frac{(a_1)_n\cdots(a_p)_n}{(b_1)_n\cdots(b_q)_n}\frac{x^n}{n!}\ee
for all $|x|\leq1$ and by continuation elsewhere.

Here $(q)_n$ is the {\it Pochhammer function} and given as follows
\begin{equation*}
(q)_n=\left\{
\begin{aligned}
1,\ \ \ \ \ \ \  &\mbox{if}\ \ \   n=0;\\
q(q+1)\cdots(q+n-1),\ \ \ \ \ \ \    &\mbox{if}\ \ \  n>0.
\end{aligned}
\right.
\end{equation*}
\end{defn}

We know from \cite{Anderson} that the following polynomials
$$\varphi(w)=\sum\limits_{n=0}^{\infty}\sum\limits_{m=0}^{\infty}a_{m, n}w^m\bar{w}^n$$
are dense in $L^2(\ID,\, \mathrm{d}A)$, where $a_{m, n}$ are complex numbers and only finite many of the complex numbers $a_{m, n}$ are nonzero (cf. \cite[Page 180]{Anderson}). Here $w=\rho e^{i\theta}\in\ID$.

Assume that $d\in\mathbb{Z}$ is an integer. By letting
\be\label{6-28-gd}g_d(\rho e^{i\theta})=f_d(\rho)e^{id\theta}=\sum_{n\geq\max\{0, -d\}}a_{n+d, n}\rho^{2n+d}e^{id\theta},\ee
the function $\varphi$ can be given as follows:
$$\varphi(w)=\sum_{d=-\infty}^{\infty}g_d(\rho e^{i\theta}).$$
We then denote by $\mathcal{F}_d$ the linear space containing all such polynomials $\varphi$.
It is easy to see that if $\varphi_i\in\mathcal{F}_i$ and $\varphi_j\in\mathcal{F}_j$, with $i\neq j$, then $\varphi_i$ and
$\varphi_j$ are orthogonal in the Hilbert space $L^2(\ID,\, \mathrm{d}A)$, and thus,
$$\|\varphi\|_2^2=\sum_{d=-\infty}^{\infty}\|g_d\|_2^2=\sum_{d=-\infty}^{\infty}2\int_0^1|f_d(r)|^2r\mathrm{d}r.$$
Throughout this paper, the functions $g_d$ and $f_d$ are always given by (\ref{6-28-gd}).

We first prove the following lemma which is related to $\mathcal{P}[g_d]$.

\begin{lem}\label{531-lem1}
Let $p(w)=a_{m, n}w^{m}\bar{w}^n$, where $w=\rho e^{i\theta}\in\ID$, $m$, $n$ are non-negative integers and $a_{m, n}$ are complex numbers. Then
for $z\in\ID$, we have
\be\label{531-lem1-1}\mathcal{P}[p](z)=\left\{
\begin{array}
{r@{\ }l}
-\frac{a_{m, n}z^{m-n-1}(1-|z|^{2n+2})}{n+1}, \ \ \ \ \ \ \    & \emph{if}\ \ \ m-n\geq1\\
\\
\frac{a_{m, n}z^{m-n-1}|z|^{2n+2}}{n+1}-\frac{\bar{a}_{m, n}z^{1-m+n}}{n+1}, \ \ \ \ \ \ &  \emph{if}\ \ \  m-n<1.
\end{array}\right.\ee
\end{lem}
\bpf
Recall that
$$\mathcal{P}[p](z)=-\int_{\ID}\left(\frac{p(w)}{w-z}+\frac{z\overline{p(w)}}{1-\bar{w}z}\right)\mathrm{d}A(w),$$
where $p(w)=a_{m, n}w^{m}\bar{w}^n\chi_{\ID}(w)$ and $z\in\ID$. Let $w=\rho e^{i\theta}\in\ID$.
We now calculate the above integrals one by one.

First, we calculate the integral $\int_{\ID}p(w)/(w-z)\mathrm{d}A(w)$ as follows:
By letting $\zeta=e^{i\theta}\in\mathbb{T}$, we have
\begin{align*}
  \int_{\ID}\frac{a_{m, n}w^{m}\bar{w}^n}{w-z}\mathrm{d}A(w) & =\frac{a_{m, n}}{\pi}\int_0^1\rho \mathrm{d}\rho\int_0^{2\pi}\frac{\rho^{m+n}e^{i(m-n) \theta}}{\rho e^{i\theta}-z}\mathrm{d}\theta \\\nonumber
   & =\frac{a_{m, n}}{\pi i}\int_0^1\rho^{m+n}\mathrm{d}\rho\int_{\mathbb{T}}\frac{\zeta^{m-n-1}}{\zeta-z/\rho}\mathrm{d}\zeta.
\end{align*}
For simplicity, let
$$I_{m-n,\, \rho}(z)=\int_{\mathbb{T}}\frac{\zeta^{m-n-1}}{\zeta-z/\rho}\mathrm{d}\zeta.$$
If $|z|>\rho$, then by using Cauchy residue theorem, we have
$$I_{m-n,\, \rho}(z)=\left\{
\begin{array}
{r@{\ }l}
0, \ \ \ \ \ \ \    & \emph{if}\ \ \ m-n\geq1\\
\\
-2\pi i\left(\frac{z}{\rho}\right)^{m-n-1}, \ \ \ \ \ \ &  \emph{if}\  \ \ m-n< 1.
\end{array}\right.$$
If $|z|<\rho$, again, by Cauchy residue theorem, we have
$$I_{m-n,\, \rho}(z)=\left\{
\begin{array}
{r@{\ }l}
2\pi i\left(\frac{z}{\rho}\right)^{m-n-1}, \ \ \ \ \ \ \    & \emph{if}\ \ \  m-n\geq1\\
\\
0, \ \ \ \ \ \ &  \emph{if}\ \ \  m-n< 1.
\end{array}\right.$$
Therefore, if $m-n\geq1$, then
\begin{align*}
  \int_{\ID}\frac{a_{m, n}w^{m}\bar{w}^n}{w-z}\mathrm{d}A(w) & =\frac{a_{m, n}}{\pi i}\int_0^{|z|}\rho^{m+n}I_{m-n,\, \rho}(z)\mathrm{d}\rho+\frac{a_{m, n}}{\pi i}\int_{|z|}^{1}\rho^{m+n}I_{m-n,\, \rho}(z)\mathrm{d}\rho \\
  & =0+2a_{m, n}\int_{|z|}^1\rho^{m+n}\left(\frac{z}{\rho}\right)^{m-n-1}\mathrm{d}\rho\\
  &=\frac{a_{m, n}z^{m-n-1}(1-|z|^{2n+2})}{n+1}.
\end{align*}
Similarly, if $m-n<1$, then
\begin{align*}
  \int_{\ID}\frac{a_{m, n}w^{m}\bar{w}^n}{w-z}\mathrm{d}A(w) &  =-2a_{m, n}\int_0^{|z|}\rho^{m+n}\left(\frac{z}{\rho}\right)^{m-n-1}\mathrm{d}\rho+0\\
  &=-\frac{a_{m, n}z^{m-n-1}|z|^{2n+2}}{n+1}.
\end{align*}
Based on the above formulas, we conclude that
\be\label{531-lem1-eq1}\int_{\ID}\frac{a_{m, n}w^{m}\bar{w}^n}{w-z}\mathrm{d}A(w)=\left\{
\begin{array}
{r@{\ }l}
a_{m, n}\frac{z^{m-n-1}(1-|z|^{2n+2})}{n+1}, \ \ \ \ \ \ \    & \emph{if}\ \ \ m-n\geq1\\
\\
-a_{m, n}\frac{z^{m-n-1}|z|^{2n+2}}{n+1}, \ \ \ \ \ \ &  \emph{if}\  \ \ m-n< 1.
\end{array}\right.\ee

Second, we calculate the integral $\int_{\ID}z\bar{p}(w)/(1-\bar{w}z)\mathrm{d}A(w)$ as follows:
Elementary calculations show that
\begin{align*}
  z\int_{\ID}\frac{\overline{a_{m, n}w^{m}\bar{w}^n}}{1-\bar{w}z}\mathrm{d}A(w) & =\frac{\bar{a}_{m, n}z}{\pi}\int_0^1\rho \mathrm{d}\rho\int_0^{2\pi}\frac{\rho^{m+n}e^{-i(m-n) \theta}}{1-\rho e^{-i\theta}z}\mathrm{d}\theta \\\nonumber
   & =\frac{\bar{a}_{m, n}z}{\pi i}\int_0^1\rho^{m+n+1}\mathrm{d}\rho\int_{\mathbb{T}}\frac{1}{\zeta^{m-n}(\zeta-\rho z)}\mathrm{d}\zeta.
\mathcal{}\end{align*}

If $m-n>0$, then
$$\frac{1}{2\pi i}\int_{\mathbb{T}}\frac{1}{\zeta^{m-n}(\zeta-\rho z)}\mathrm{d}\zeta=\mbox{Res}\left(\frac{1}{\zeta^{m-n}(\zeta-\rho z)}, 0\right)+\mbox{Res}\left(\frac{1}{\zeta^{m-n}(\zeta-\rho z)}, \rho z\right)=0,$$
and thus
$$z\int_{\ID}\frac{\overline{a_{m, n}(\bar{w}^mw^{n})}}{1-\bar{w}z}\mathrm{d}A(w) =0.$$
If $m-n\leq0$,
then
$$\frac{1}{2\pi i}\int_{\mathbb{T}}\frac{\zeta^{n-m}}{\zeta-\rho z}\mathrm{d}\zeta=\mbox{Res}\left(\frac{\zeta^{n-m}}{\zeta-\rho z}, \rho z\right)=(\rho z)^{n-m},$$
and thus
$$z\int_{\ID}\frac{\overline{a_{m, n}w^{m}\bar{w}^n}}{1-\bar{w}z}\mathrm{d}A(w)=2\bar{a}_{m, n}z\int_0^1\rho^{m+n+1}(\rho z)^{n-m}\mathrm{d}\rho =\frac{\bar{a}_{m, n}z^{1-m+n}}{n+1}.$$
Based on the above formulas, we have
\be\label{531-lem1-eq2}z\int_{\ID}\frac{\overline{a_{m, n}w^{m}\bar{w}^n}}{1-\bar{w}z}\mathrm{d}A(w)=\left\{
\begin{array}
{r@{\ }l}
0, \ \ \ \ \ \ \    & \emph{if}\ \ \ m-n\geq1\\
\\
\frac{\bar{a}_{m, n}z^{1-m+n}}{n+1}, \ \ \ \ \ \ &  \emph{if}\ \ \  m-n<1.
\end{array}\right.\ee
The desired equality (\ref{531-lem1-1}) now easily follows from (\ref{531-lem1-eq1}) and (\ref{531-lem1-eq2}).
\epf
\subsection{The functions $\mathcal{P}[g_{d_1}]$ and $\mathcal{P}[g_{d_2}]$ are orthogonal in $L^2(\ID,\, \mathrm{d}A)$ iff $d_1+d_2\neq 2$}
\begin{lem}\label{orth1-lemma}
Let $g_d$ be given by $(\ref{6-28-gd})$.
Then for any different integers $d_1$, $d_2$, so that  $d_1+d_2\neq 2$,
$\mathcal{P}[g_{d_1}](z)$ and $\mathcal{P}[g_{d_2}](z)$ are orthogonal functions, i.e.,
\be\label{orthogonal-1}\int_{\ID}\mathcal{P}[g_{d_1}](z)\overline{\mathcal{P}[g_{d_2}](z)}\mathrm{d}A(z)=0\ \ \ \mbox{for any} \ \ \ d_1\neq d_2.\ee
However, for $d\geq2$, $\mathcal{P}[g_{d}](z)$ and $\mathcal{P}[g_{2-d}](z)$ are not orthogonal functions.
\end{lem}
\begin{proof}
Fix $z=re^{it}\in\ID$. It follows from (\ref{531-lem1-1}) in Lemma \ref{531-lem1} that
\be\label{pgd}\mathcal{P}[g_d](z)=\left\{
\begin{array}
{r@{\ }l}
-\sum\limits_{n=0}^{\infty}\frac{b_n(1-|z|^{2n+2})z^{d-1}}{n+1}, \ \ \ \ \ \ \    & \mbox{if}\ \ \ d\geq1\\
\\
\sum\limits_{n=-d}^{\infty}\left(\frac{b_n|z|^{2n+2}z^{d-1}}{n+1}-\frac{\bar{b}_nz^{1-d}}{n+1}\right), \ \ \ \ \ \ &  \mbox{if}\ \ \  d<1.
\end{array}\right.\ee

Let $d_1$ and $d_2$ be two different integers. We should divide our proof into three cases.
\begin{description}
  \item[Case 1] $d_1$, $d_2\geq1$ and $d_1\neq d_2$.
\end{description}
Elementary calculations show that
$$ \mathcal{P}[g_{d_1}](z)\overline{\mathcal{P}[g_{d_2}](z)}=\sum_{n,\, l\geq0}\frac{b_{n}\bar{b}_l}{(n+1)(l+1)}(1-r^{2n+2})(1-r^{2l+2})r^{d_1+d_2-2}e^{i(d_1-d_2)t}.$$
It is easy to see that $\int_{\ID}\mathcal{P}[g_{d_1}](z)\overline{\mathcal{P}[g_{d_2}](z)}\mathrm{d}A(z)=0$,
since $\int_0^{2\pi}e^{ikt}\mathrm{d}t=0$ for any $k\neq0$.

\begin{description}
  \item[Case 2] $d_1$, $d_2<1$ and $d_1\neq d_2$.
\end{description}
Elementary calculations show that
\begin{align*}
  \mathcal{P}[g_{d_1}](z)\overline{\mathcal{P}[g_{d_2}](z)}& =\sum_{n\geq -d_1,\, l\geq-d_2}\frac{1}{(n+1)(l+1)}\bigg\{b_{n}\bar{b}_lr^{2n+2l+d_1+d_2+2}e^{i(d_1-d_2)t}\\
  &\ \ \ - b_nb_lr^{2n+d_1-d_2+2}e^{i(d_1+d_2-2)t} \\
  &\ \ \ -\bar{b}_{n}\bar{b}_lr^{2l+d_2-d_1+2}e^{i(2-d_1-d_2)t}\\
   &\ \ \ +\bar{b}_nb_lr^{2-d_1-d_2}e^{i(d_2-d_1)t}\bigg\}.\\
\end{align*}
It follows from the assumption of this case that $d_1+d_2\neq2$ and $d_1-d_2\neq0$. Therefore (\ref{orthogonal-1}) holds true,
and again, we have $\mathcal{P}[g_{d_1}](z)$ and $\mathcal{P}[g_{d_2}](z)$ are orthogonal in this case.

\begin{description}
  \item[Case 3] $d_1\geq1$ and $d_2<1$.
\end{description}
Elementary calculations show that
\begin{align}\label{61-1-eq2}
  \mathcal{P}[g_{d_1}](z)\overline{\mathcal{P}[g_{d_2}](z)}& =-\sum_{n\geq0,\, l\geq-d_2}\left\{b_n\bar{b}_l\frac{(1-r^{2n+2})r^{2l+d_1+d_2}}{(n+1)(l+1)}e^{i(d_1-d_2)t}\right. \\\nonumber
  &\ \ \ \left.-b_nb_l\frac{(1-r^{2n+2})r^{d_1-d_2}}{(n+1)(l+1)}e^{i(d_1+d_2-2)t}\right\}.
\end{align}

If $d_1+d_2-2\neq0$, then $\int_{\ID}\mathcal{P}[g_{d_1}](z)\overline{\mathcal{P}[g_{d_2}](z)}\mathrm{d}A(z)=0$.

If $d_1+d_2-2=0$, then
\begin{align}\label{61-1-eq3}
  \nonumber\int_{\ID}\mathcal{P}[g_{d_1}](z)\overline{\mathcal{P}[g_{d_2}](z)}\mathrm{d}A(z) & = 2\sum_{n\geq0,\, l\geq-d_2}\frac{b_n{b}_l}{(n+1)(l+1)}\int_0^1(1-r^{2n+2})r^{d_1-d_2+1}\mathrm{d}r\\
   &=\sum_{n\geq0,\, l\geq-d_2}\frac{b_n{b}_l}{(2-d_2)(n+3-d_2)(l+1)}.
\end{align}
This shows that $\mathcal{P}[g_{d_1}](z)$ and $\mathcal{P}[g_{d_2}](z)$ are not orthogonal functions in this case.
\end{proof}

\subsection{A representation formula for  $\mathcal{P}[g_d]$.}

\begin{lem}\label{lem-1-cgd-conju}
Suppose $w=\rho e^{i\theta}\in\ID$ and let $g_d(w)=f_d(\rho)e^{i{d}\theta}$ be given by $(\ref{6-28-gd})$, where $d$ is an integer. Then
\be\label{611-3}\mathfrak{C}[g_d](z)=\left\{
\begin{array}
{r@{\ }l}
-2z^{d-1}\int_0^{|z|}\frac{f_d(\rho)}{\rho^{d-1}}\mathrm{d}\rho\ , \ \ \ \   &\mbox{for}\ \  \ d\leq 0\\
\\
2z^{d-1}\int_{|z|}^1\frac{f_d(\rho)}{\rho^{d-1}}\mathrm{d}\rho, \ \ \ \ &\mbox{for}\ \ \  d\geq1.
\end{array}\right.\ee
\end{lem}
\begin{proof}
Recall that the Cauchy integral operator is given as follows:
$$\mathfrak{C}[g_d](z)=\int_{\ID}\frac{g_d(w)}{w-z}\mathrm{d} A(w).$$
Let $w=\rho e^{i\theta}\in\ID$. Then
\begin{align}\label{cau-conjugate-622}
 \mathfrak{C}[g_d](z) & =\frac{1}{\pi}\int_0^1\rho \mathrm{d}\rho\int_0^{2\pi}\frac{f_d(\rho)e^{i\mathrm{d}\theta}}{\rho e^{i\theta}-z}\mathrm{d}\theta \\\nonumber
   &=\frac{1}{\pi }\int_0^{|z|} f_d(\rho)I_{d,\, \rho}(z)\mathrm{d}\rho+\frac{1}{\pi}\int_{|z|}^1 f_d(\rho)I_{d,\, \rho}(z)\mathrm{d}\rho,
\end{align}
where
$$I_{d,\, \rho}(z)=\int_0^{2\pi} \frac{e^{id\theta}}{e^{i\theta}-z/\rho}\mathrm{d}\theta.$$

We now calculate $I_{d,\, \rho}(z)$ as follows:
Let $\zeta=e^{it}\in\mathbb{T}$. It follows from Cauchy residue theorem that
$$I_{d,\, \rho}(z)=\frac{1}{i}\int_{|\zeta|=1}\frac{\zeta^{d-1}}{\zeta-z/\rho}\mathrm{d}\zeta.$$

If $|z|/\rho>1$, then for $d\geq1$, we have $I_{d,\, \rho}(z)=0$, while for $d\leq0$, we have
$$I_{d,\, \rho}(z)=2\pi\mbox{Res}\left[\frac{1}{(\zeta-z/\rho)\zeta^{1-d}},0\right]=-2\pi\left(\frac{z}{\rho}\right)^{d-1}.$$

If $|z|/\rho<1$, then for $d\geq1$, we have $I_{d,\, \rho}(z)=2\pi\left(z/\rho\right)^{d-1}$,
while for $d\leq0$, we have
$$ I_{d,\, \rho}(z) =2\pi\left(\mbox{Res}\left[\frac{1}{(\zeta-z/\rho)\zeta^{1-d}},0\right]+\mbox{Res}\left[\frac{1}{(\zeta-z/\rho)\zeta^{1-d}},z/\rho\right]\right)=0.$$
Based on the above discussions together with (\ref{cau-conjugate-622}), we know (\ref{611-3}) holds true.
\end{proof}

\begin{lem}\label{lem-2-jgd-conju}
Suppose $w=\rho e^{i\theta}\in\ID$ and let $g_d(w)=f_d(\rho)e^{i{d}\theta}$ be given by $(\ref{6-28-gd})$, where $d$ is an integer. Then
\be\label{611-5}\mathfrak{J}_0[\overline{g_d}](z)=\left\{
\begin{array}
{r@{\ }l}
2z^{1-d}\int_0^{1}\rho^{1-d}\bar{f}_d(\rho)\mathrm{d}\rho\ , \ \ \ \ &\mbox{for}\ \ \ d\leq 0\\
\\
0, \ \ \ \ & \mbox{for}\ \ \  d\geq1.
\end{array}\right.\ee
\end{lem}
\begin{proof}
Recall that
$$\mathfrak{J}_0[g_d](z)=\int_{\ID}\frac{z}{1-\bar{w}z}g_d(w)\mathrm{d}A(w).$$

Suppose $w=\rho e^{i\theta}\in\ID$. Then
\begin{align}\label{622-3-conu-0}
 \mathfrak{J}_0[\overline{g_d}](z) & =\frac{z}{\pi}\int_0^1\rho \mathrm{d}\rho\int_0^{2\pi}\frac{\bar{f}_d(\rho)e^{-id\theta}}{1-\rho e^{-i\theta}z}\mathrm{d}\theta \\\nonumber
   & =\frac{z}{\pi}\int_0^1\bar{f}_d(\rho)J_{d, \rho}(z)\rho \mathrm{d}\rho,
\end{align}
where
$$J_{d,\, \rho}(z)=\int_0^{2\pi}\frac{e^{-id\theta}}{1-\rho e^{-i\theta}z}\mathrm{d}\theta.$$
Let $\zeta=e^{it}\in\mathbb{T}$. It follows from Cauchy residue theorem that
$$J_{d,\, \rho}(z)=\frac{1}{i}\int_{|\zeta|=1}\frac{\zeta^{-d}}{\zeta-\rho z}\mathrm{d}\zeta.$$
Sice $z\in \ID$, elementary calculations show that
\be\label{625-4}J_{d,\, \rho}(z)=\left\{
\begin{array}
{r@{\ }l}
2\pi (\rho z)^{-d},\ \ \  &\mbox{for}\ \ \ d\leq 0\\
\\
0,\ \ \  &\mbox{for}\ \ \ d\geq1.
\end{array}\right.\ee
The desired equality (\ref{611-5}) now easily follows from (\ref{622-3-conu-0}) and (\ref{625-4}).
\end{proof}

According to Lemma \ref{lem-1-cgd-conju} and Lemma \ref{lem-2-jgd-conju}, we see that
\be\label{ope-P-gd}\mathcal{P}[g_d](z)=\left\{
\begin{array}
{r@{\ }l}
2z^{d-1}\int_0^{|z|}\rho^{1-d}f_d(\rho)\mathrm{d}\rho-2z^{1-d}\int_0^{1}\rho^{1-d}\bar{f}_d(\rho)\mathrm{d}\rho, \ \ \ & \mbox{for}\ \  d\leq 0\\
\\
-2z^{d-1}\int_{|z|}^1\rho^{1-d}f_d(\rho)\mathrm{d}\rho, \ \ \ & \mbox{for}\ \  d\geq1.
\end{array}\right.\ee

\subsection{Hardy's type inequalities}

Suppose $m\in C^1(a, b)$, is a continuously differentiable function in $(a, b)$. Denote by $L_m^p(a, b)$ the space of functions $h$ defined in $(a, b)$ which satisfy the following condition
$$\|h\|_p=\left(\int_a^b|h(t)|^pm(t)\mathrm{d}t\right)^{1/p}<\infty, \ \ \ \mbox{for}\ \ \ 0<p<\infty.$$
The following result was due to Boyd (see \cite[Theorem 1]{Boyd}), which is useful in proving our Theorem \ref{general-2norm}.

\begin{Thm}\label{Boya}$($\cite[Theorem 1]{Boyd}$)$
Suppose that $w(x)$, $m(x)\in C^1(a, b)$, that $w(x)>0$ a.e. and $m(x)>0$ for $a<x<b$, that $p>0$, $r>1$, $0\leq q<r$, and that the following operator
\be\label{630-ope-T}T[h](x)=w(x)^{1/p}m(x)^{-1/p}\int_a^xh(t)\mathrm{d}t,\ee
is compact from $L_m^r(a, b)\to L_m^s(a, b)$ $(s=pr/(r-q))$. Then the following eigenvalue problem $(\ref{boyd-pde})$ has solutions ($y, \lambda$) with $y\in C^2(a, b)$ and
$y(x)>0$, $y'(x)>0$ in $(a, b)$.
\be\label{boyd-pde}\left\{
\begin{array}
{r@{\ }l}
&(i)\ \  \frac{d}{dx}\bigg(r\lambda (y')^{r-1}m-qy^py'^{q-1}w\bigg) +py^{p-1}(y')^{q}w=0; \\
\\
&(ii)\ \ \lim_{x\to a}y(x)=0\ \ \ \mbox{and}\ \ \ \lim_{x\to b} (r\lambda (y')^{r-1}m-qy^p(y')^{q-1}w)=0;\\
\\
&(iii)\ \ \|y'\|_r=1.
\end{array}\right.\ee
There is a largest value $\lambda$ such that $(\ref{boyd-pde})$ has a solution and if $\lambda^*$ denotes this value, then for any $h\in L_m^r(a, b)$,
\be\label{boyd-eq1}
\int_a^b\left|\int_a^xh(t)\mathrm{d}t\right|^p|h(x)|^qw(x)\mathrm{d}x\leq\frac{r\lambda^*}{p+q}\left(\int_a^b|h(x)|^rm(x)\mathrm{d}x\right)^{(p+q)/r}.
\ee
Equality holds in $(\ref{boyd-eq1})$ if and only if $h=cy'$ a.e. where $y$ is a solution of $(\ref{boyd-pde})$ corresponding to $\lambda=\lambda^*$, and $c$ is any constant.
\end{Thm}

The following lemma is one of the main tools in proving our result.

\begin{lem}\label{lem-628-boyd-d}
Let $d$ be an integer and suppose $u\in L_m^2(0, 1)$ is a real or complex function, where $m(x) = x$. Then
for $d\leq0$, we have
\be\label{be-j1}\int_0^1\left|\int_0^r \rho^{1-d}u(\rho)\mathrm{d}\rho\right|^2r^{2d-1}\mathrm{d}r\leq \frac{1}{j_{-d}^2}\int_0^1\rho |u(\rho)|^2\mathrm{d}\rho,\ee
and for $d\geq1$, we have
\be\label{be-j2}\int_0^1\left|\int_r^1\rho^{1-d}u(\rho)\mathrm{d}\rho\right|^2r^{2d-1}\mathrm{d}r\leq \frac{1}{j_{d-1}^2}\int_0^1\rho |u(\rho)|^2\mathrm{d}\rho,\ee
where $j_d$ is the smallest positive zero of the Bessel function $J_d(x)$.
\end{lem}
\bpf
Suppose $d\leq0$. To prove (\ref{be-j1}), we will use Theorem \Ref{Boya} to find the best constant $\lambda$ such that
$$\int_0^1\left|\int_0^r \rho^{1-d}u(\rho)\mathrm{d}\rho\right|^2r^{2d-1}\mathrm{d}r\leq \lambda\int_0^1\rho |u(\rho)|^2\mathrm{d}\rho,$$
where $\lambda=\frac{1}{j_{-d}^2}$.

In our case, $a=0$, $b=1$, $p=r=2$, $q=0$, $h(x)=x^{1-d}u(x)$, $w(x)=x^{2d-1}=m(x)$. The equality holds if $h=cy'$, where
$y$ is a solution to the following ordinary differential equation:
$$\frac{d}{dx}(2\lambda y'(x)x^{2d-1})+2y(x)x^{2d-1}=0.$$
The general solution to the previous ODE is $$h(x) = c_1 x^{1-d} J_{1-d} \left(\frac{x}{\sqrt{\lambda}}\right)+c_2 x^{1-d}\mathcal{N}_{d-1}\left(\frac{x}{\sqrt{\lambda}}\right),$$
where $J_{1-d}$ and $\mathcal{N}_{d-1}$ are Bessel functions given by (\ref{Bessel-function}) and (\ref{Bessel-function2}), respectively.
Therefore, in view of the initial condition $h(0)=0$ (by Theorem~A), we have
$$h(x)=c x^{1-d}J_{1-d}\left(\frac{x}{\sqrt{\lambda}}\right),$$
where $c>0$ is a constant.
This shows that in this case the best constant is $\lambda=\frac{1}{j_{-d}^2}$, where $j_{-d}$ is the smallest positive zero of $J_{-d}(x)$.

Moreover, in order to use Theorem \Ref{Boya}, we should prove that  the operator $T$, which is defined by (\ref{630-ope-T}), is compact.
A simple sufficient condition for $T$ to be compact from $L_m^r(a, b)\to L_m^{s'}(a, b)$, where $s'=pr/(r-q)$, is that
$$\|T\|=\bigg\{\int_a^b\left(\int_a^bk(x, t)^{r'}m(t)\mathrm{d}t\right)^{s/r'}m(x)\mathrm{d}x\bigg\}^{1/s}<\infty, \ \ \ s>1,\ \ \ r>1,$$
where $r'=\frac{r}{r-1}$, $s=\frac{s'}{s'-1}$ and
$$k(x, t)=w(x)^{1/p}m(x)^{-1/p}m(t)^{-1}\chi_{[a, x]}(t).$$
For this argument we refer the readers to \cite[Page 319]{Zaa}.

In our case, we have $p=s=s'=r=r'=2$, $a=0$, $b=1$, $w(x)=x^{2d-1}=m(x)$, and thus,
$$\|T\|=\bigg\{\int_0^1\left(\int_0^x(t^{2d-1})^{-1}\mathrm{d}t\right)x^{2d-1}\mathrm{d}x\bigg\}^{1/2}=\left(\frac{1}{4(1-d)}\right)^{1/2}<\infty,$$
since $d\leq0$. Based on the above observations, we have (\ref{be-j1}) holds true.

Next, we prove (\ref{be-j2}) as follows: It is easy to see that (\ref{be-j2}) is equivalent to
$$\int_0^1 \left|\int_0^{r} \frac{u(1-\rho)}{(1-\rho)^{d-1}} \mathrm{d}\rho\right|^2(1-r)^{2d-1} \mathrm{d}r\le \frac{1}{j_d^2}\int_0^1(1-\rho) |u(1-\rho)|^2\mathrm{d}\rho.$$
Again, we will use Theorem \Ref{Boya} to prove the above inequality.

In this case, $a=0$, $b=1$, $p=r=2$, $q=0$, $H(x)=(1-x)^{1-d}u(1-x)$, $W(x)=(1-x)^{2d-1}=M(x)$. We now seek the best constant $\lambda$, such that the following inequality
$$\int_0^1 \left|\int_0^{r} H(\rho) \mathrm{d}\rho\right|^2W(r)^{2d-1} \mathrm{d}r\le \lambda\int_0^1|H(\rho)|^2M(\rho)\mathrm{d}\rho,$$
holds. This can be done by letting $H=cy'$, where
$y$ is a solution of the following ordinary differential equation:
$$\frac{d}{dx}(2\lambda y'(x)(1-x)^{2d-1})+2y(x)(1-x)^{2d-1}=0.$$
Then, again by having in mind the initial condition $y(0)=0$ imposed by Theorem~A,
similar to the proof of the  previous part,  we have
$$y(x)=\frac{J_{d-1}(1/\sqrt{\lambda}-x/\sqrt{\lambda})}{(1-x)^d}.$$
Therefore, $\lambda=\frac{1}{j_{d-1}^2}$, where $j_{d}$ is the smallest positive zero of $J_{d}(x)$.

Finally, we prove the compactness of the operator $T$ as follows: For $d\geq1$,
$$\|T\|=\bigg\{\int_0^1\left(\int_0^x\big[(1-t)^{2d-1}\big]^{-1}\mathrm{d}t\right)(1-x)^{2d-1}\mathrm{d}x\bigg\}^{1/2}=\frac{1}{2\sqrt{d}}<\infty,$$
which shows that $T$ is compact. Therefore, by using Theorem \Ref{Boya}, we see that (\ref{be-j2}) holds true.
The proof is complete.
\epf

For some non-negative integers $d$, we list some (approximate) values of $j_d$ as follows:
\begin{table}[h]
\centerline{\large\begin{tabular}{|c|c|c|c|c|c|c|}
\hline
\ $d$\ &\ 0\ &\ 1\ &\ 2\ &\ 3\ &\ 4\ \\
\hline
\ $j_d$\ &\ 2.4048 \ &\ 3.8317\ &\  5.1356\ &\ 6.3802\ &\ 7.5883\  \\
\hline
\end{tabular}}
\end{table}

The following lemma will be used in the proof of Theorem \ref{general-2norm}, Step 4.
\subsection{Existence of the largest eigenvalue for operator $\mathcal{P}^*\mathcal{P}$}

\begin{lem}\label{step-4-lem6-7-9}
Let $\mathcal{Y}_0$ be a closed linear subspace of $L^2(\ID, \mathrm{d}A)$. Then there exists a constant $\lambda_\circ$ and a mapping $f\in \mathcal{Y}_0$, such that
$$\sup_{\varphi\in \mathcal{Y}_0\setminus\{0\}}\frac{\|\mathcal{P} [\varphi]\|_2^2}{\|\varphi\|_2^2}=\lambda_\circ^2=\frac{\|\mathcal{P} [f]\|_2^2}{\|f\|_2^2}.$$
Moreover, $\lambda_\circ$ is the largest eigenvalue of the operator $T=\mathcal{P}^* \mathcal{P}$. In other words, $$\lambda_\circ=\max\big\{\lambda:\exists\, \varphi\in \mathcal{Y}_0, \ s.t. \ \mathcal{P}^* \mathcal{P} \varphi=\lambda \varphi\big\}.$$
\end{lem}
\begin{proof}
Suppose $f\in \mathcal{Y}_0$ and let $\mu =\sup_{\|f\|=1}\left\langle T [f],f\right\rangle$. Then
$$\mu=\sup_{\|f\|=1}\langle \mathcal{P}[f], \mathcal{P}[f]\rangle=\sup_{\|f\|=1}\|\mathcal{P}[f]\|_2^2 >0.$$
Now, assume that $\{f_n\}_{n=1}^{\infty}$ is a sequence of mappings such that $\|f_n\|_2=1$ and $\lim_{n\to \infty} \left\langle T f_n,f_n\right\rangle=\mu$.
Note that $T$ is a compact operator, we have there is a subsequence  $\{f_{n_k}\}_{k=1}^{\infty}$, such that $Tf_{n_{k}} \to g$, where $g\in \mathcal{Y}_0\subset L^2(\ID, \mathrm{d}A)$.

Let $g=\mu f$.
Then
\begin{align*}
  \lim_{n\to \infty }\|(T -\mu I)[f_n] \|_2^2& =\lim_{n\to \infty }\bigg(\left<T[f_n], T[f_n]\right> -2\mu \left<T[f_n],f_n\right> +\mu^2 \|f_n\|_2^2\bigg) \\
   & \leq\mu^2-2\mu^2+\mu^2=0.
\end{align*}

On the other hand, since
 $$\|\mu(f_n-f)\|_2\le \|(T-\mu I)[f_n]\|_2+\|T[f_n]-\mu f\|_2,$$
we see that
$$\|f_n-f\|_2\le \mu^{-1}\bigg(\|(T-\mu I)[f_n]\|_2+\|T[f_n]-\mu f\|_2\bigg),$$
which shows $\lim_{n\to \infty}f_n=f$.

Based on the above observations, the following equalities hold
$$\lim_{n\to \infty } T[f_n] = T[f] = \mu f,$$
because the operator $T$ is continuous.
The assumption of $\mu$ at the beginning ensures that $\mu =\lambda_\circ^2$, the largest eigenvalue of the operator $T=\mathcal{P}^* \mathcal{P}$,
and thus, the proof is complete.
\end{proof}
\section{Proofs of Theorems \ref{general-2norm} $-$  \ref{infinity-norm}}\label{sec-3}
\subsection*{Proof of Theorem \ref{general-2norm}}
Recall that for $w=\rho e^{i\theta}\in\ID$, the following functions
$$\varphi(w)=\sum\limits_{n=0}^{\infty}\sum\limits_{m=0}^{\infty}a_{m, n}w^m\bar{w}^n=\sum_{d=-\infty}^{\infty}g_d(\rho e^{i\theta})$$
are dense in $L^2(\ID,\, \mathrm{d}A)$, where
$$g_d(\rho e^{i\theta})=f(\rho)e^{id\theta}=\sum_{n\geq\max\{0, -d\}}a_{n+d, n}\rho^{2n+d}e^{id\theta}.$$
Then
$$\|\varphi\|_2^2=\sum_{d=-\infty}^{\infty}\|g_d\|_2^2,$$
where
\be\label{gd-def}\|g_d\|_2^2=2\int_0^1|f(\rho)|^2r\mathrm{d}\rho.\ee
Following the proof of \cite[Theorem 1]{Anderson}, if one can find the best constant $M$ such that
$$\|\mathcal{P}[\varphi]\|_2^2\leq M\|\varphi\|_2^2,$$
then $\|\mathcal{P}\|_2^2\leq M$. Furthermore, if there is an extremal function $\varphi_0$ such that $\|\mathcal{P}[\varphi_0]\|_2^2=M\|\varphi_0\|_2^2$,
then $\|\mathcal{P}\|_2^2=M.$

In what follows, we will use Hardy's type inequalities, Bessel functions and Lagrange multiplier methods to find such a constant $M$.
We divide our proof into four steps.

\begin{description}
  \item[Step 1]  {\bf The representation formula for  $\|\mathcal{P}[\varphi]\|_2^2$.}
\end{description}

According to Lemma \ref{orth1-lemma}, we see that
\begin{align*}
  \|\mathcal{P}[\varphi]\|_2^2 & =\int_{\ID}\left|\sum_{d=-\infty}^{\infty}\mathcal{P}[g_d](z)\right|^2dA(z) \\
   & =\sum_{d=-\infty}^{\infty}\|\mathcal{P}[g_d]\|_2^2+2\mbox{Re}\sum_{d\geq2}\langle\mathcal{P}[g_d], \mathcal{P}[g_{2-d}]\rangle.
\end{align*}
Since $\mathcal{P}[\varphi]=-\mathfrak{C}[\varphi]-\mathfrak{J}_0[\bar{\varphi}]$, it follows from Lemma 2.4 that
$$\mathcal{P}[g_d]=-\mathfrak{C}[g_d],\ \ \ \ \mbox{for}\ \ \  d\geq1.$$
Therefore, again by using Lemma \ref{orth1-lemma}, we have
\be\label{630-eq1}\|\mathcal{P}[\varphi]\|_2^2 =\sum_{d\leq0}\|\mathcal{P}[g_d]\|_2^2+\sum_{d\geq1}\|\mathfrak{C}[g_d]\|_2^2-2\mbox{Re}\sum_{d\geq2}\langle\mathfrak{C}[g_d], \mathcal{P}[g_{2-d}]\rangle,\ee
where
\be\label{630-eq2}\sum_{d\leq0}\|\mathcal{P}[g_d]\|_2^2=\sum_{d\leq0}\|\mathfrak{C}[g_d]\|_2^2+\sum_{d\leq0}\|\mathfrak{J}_0[\bar{g}_d]\|_2^2.\ee

Let $z=r e^{it}\in \ID$. The following two equalities are due to Lemma \ref{lem-1-cgd-conju}
\be\label{cauchy-gd-1}\|\mathfrak{C}[g_d](z)\|_2^2=8\int_0^1\left|\int_0^r \rho^{1-d}f_d(\rho)\mathrm{d}\rho \right|^2r^{2d-1}\mathrm{d}r,\ \ \ \mbox{for}\ \ \ d\leq0,\ee
and
\be\label{cauchy-gd-2}\|\mathfrak{C}[g_d](z)\|_2^2=8\int_0^1\left|\int_r^1 \rho^{1-d}f_d(\rho)\mathrm{d}\rho \right|^2r^{2d-1}\mathrm{d}r,\ \ \ \mbox{for}\ \ \ d\geq1,\ee
while Lemma \ref{lem-2-jgd-conju} gives the following equality
\be\label{j0-gd-628}\|\mathfrak{J}_0[\bar{g}_d](z)\|_2^2=\frac{4}{2-d}\left|\int_0^1\rho^{1-d}f_d(\rho)\mathrm{d}\rho\right|^2,\ \ \ \mbox{for}\ \ \ d\leq0.\ee
Moreover, for $d\geq2$, one can deduce from (\ref{611-3}) and (\ref{ope-P-gd}) that
$$\langle\mathfrak{C}[g_d](z), \mathcal{P}[g_{2-d}](z)\rangle=-8\left(\int_0^1\rho^{d-1}{f}_{2-d}(\rho)\mathrm{d}\rho\right)\int_0^1\left(\int_r^1 \rho^{1-d}f_d(\rho)\mathrm{d}\rho\right)r^{2d-1}\mathrm{d}r.$$

\begin{description}
  \item[Step 2] {\bf Estimation of $\|\mathcal{P}[\varphi]\|_2^2$.}
\end{description}

Since
$$2|\mbox{Re}\langle\mathfrak{C}[g_d], \mathcal{P}[g_{2-d}]\rangle|\leq2\mathfrak{C}[g_d]\|_2\|\mathcal{P}[g_{2-d}]\|_2\leq\|\mathfrak{C}[g_d]\|_2^2+\|\mathcal{P}[g_{2-d}]\|_2^2,$$
it follows from (\ref{630-eq1}) and (\ref{630-eq2}) that
\begin{align}\label{pphi-628}
  \|\mathcal{P}[\varphi]\|_2^2 & \leq \|\mathcal{P}[g_0]\|_2^2+\|\mathfrak{C}[g_2]\|_2^2-2\mbox{Re}\langle \mathfrak{C}[g_2], \mathcal{P}[g_0]\rangle\\\nonumber
  &\ \ \ +2\sum_{d\leq-1}\|\mathcal{P}[g_d]\|_2^2+\|\mathfrak{C}[g_1]\|_2^2+2\sum_{d\geq3}\|\mathfrak{C}[g_d]\|_2^2\\\nonumber
  &=\|\mathcal{P}[g_0+g_2]\|_2^2+2\sum\limits_{\substack{d\leq-1\\ d\geq3}}\|\mathfrak{C}[g_d]\|_2^2+\|\mathfrak{C}[g_1]\|_2^2+2\sum_{d\leq-1}\|\mathfrak{J}_0[\bar{g}_d]\|_2^2.
\end{align}

For $d\geq3$, we see from (\ref{gd-def}), (\ref{cauchy-gd-2}) and (\ref{be-j2}) that
$$\|\mathfrak{C}[g_d]\|_2^2\leq \frac{4}{j_{d-1}^2}\|g_d\|_2^2,$$
where $j_d$ is the smallest positive zero of the Bessel function $J_d(x)$ and $4/j_{d-1}^2\leq4/j_2^2\approx0.152.$

While for $d\leq -1$, the equalities (\ref{gd-def}), (\ref{cauchy-gd-1}) and the sharp inequality (\ref{be-j1}) show that
$$\|\mathfrak{C}[g_d]\|_2^2\leq \frac{4}{j_{-d}^2}\|g_d\|_2^2,$$
where $j_{-d}$ is the smallest positive zero of the Bessel function $J_{-d}(x)$ and $4/j_{-d}^2\leq4/j_{1}^2\approx0.272.$

For the case $d=1$, one can easily deduce from (\ref{gd-def}), (\ref{cauchy-gd-2}) and (\ref{be-j2}) that
$$\|\mathfrak{C}[g_1]\|_2^2\leq \frac{4}{j_0^2}\|g_1\|_2^2.$$
To estimate $\|\mathfrak{J}_0[\bar{g}_d]\|_2$, by using H\"older's inequality for integral and (\ref{gd-def}), (\ref{j0-gd-628}), we have for $d\leq-1$,
the following inequalities hold
$$\|\mathfrak{J}_0[\bar{g}_d]\|_2^2=\frac{4}{2-d}\left|\int_0^1\rho^{1-d}f_d(\rho)\mathrm{d}\rho\right|^2\leq\frac{4}{(2-2d)(2-d)}\|g_d\|_2^2\leq\frac{1}{3}\|g_d\|_2^2.$$
Now, by letting
$$M=\max\bigg\{\sup\frac{\|\mathcal{P}[g_0+g_2]\|_2^2}{\|g_0\|_2^2+\|g_2\|_2^2}, \ \ \frac{8}{j_2^2},\ \  \frac{8}{j_1^2},\ \  \frac{4}{j_0^2},\ \ \frac{1}{3}\bigg\},$$
we then have $\|\mathcal{P}[\varphi]\|_2^2\leq M\|\varphi\|_2^2$, and thus $\|\mathcal{P}\|_2^2\leq M$.

In what follows, we will show that, in fact
$$M=\sup\frac{\|\mathcal{P}[g_0+g_2]\|_2^2}{\|g_0\|_2^2+\|g_2\|_2^2}.$$
Moreover, we will find an extremal function $\varphi_0\neq 0$ such that $\|\mathcal{P}[\varphi_0]\|_2^2=M\|\varphi_0\|_2^2$.

\begin{description}
  \item[Step 3] {\bf The proof of}
  $$M\in\left(\frac{4}{j_0^2}, \frac{3}{2}+\frac{2}{j_1^2}\right).$$
\end{description}

It follows from (\ref{ope-P-gd}) that
$$\mathcal{P}[g_0+g_2]=2z^{-1}\int_0^{|z|}\rho f_0(\rho)\mathrm{d}\rho-2z\int_0^1\rho\bar{f}_0(\rho)\mathrm{d}\rho-2z\int_{|z|}^1\rho^{-1}f_2(\rho)\mathrm{d}\rho,$$
and thus
\begin{align}\label{july-4-eq1}
  \|\mathcal{P}[g_0+g_2]\|_2^2& =8\int_0^1\left|\int_0^r \rho f_0(\rho)\mathrm{d}\rho \right|^2r^{-1}\mathrm{d}r+8\int_0^1\left|\int_r^1 \rho^{-1}f_2(\rho)\mathrm{d}\rho \right|^2r^{3}\mathrm{d}r\\\nonumber
   &\ \ \ +2\left|\int_0^1\rho f_0(\rho)\mathrm{d}\rho\right|^2\\\nonumber
   &\ \ \ +16\mbox{Re}\left(\int_0^1\rho{f}_{0}(\rho)\mathrm{d}\rho\right)\int_0^1\left(\int_r^1 \rho^{-1}f_2(\rho)\mathrm{d}\rho\right)r^{3}\mathrm{d}r,
\end{align}
where
$$f_0(\rho)=\sum_{n\geq0}a_{n, n}\rho^{2n}\ \ \ \mbox{and}\ \ \ f_2(\rho)=\sum_{n\geq0}a_{n+2, n}\rho^{2n+2}.$$

We now give a quick estimate for (\ref{july-4-eq1}) as follows:
By using H\"older's inequality for integral and (\ref{be-j2}), we have
\begin{align*}
  &16\mbox{Re}\left(\int_0^1\rho{f}_{0}(\rho)\mathrm{d}\rho\right)\int_0^1\left(\int_r^1 \rho^{-1}f_2(\rho)\mathrm{d}\rho\right)r^{3}\mathrm{d}r\\
   & \leq 16\left(\frac{1}{2}\int_0^1\rho|f_0(\rho)|^2\mathrm{d}\rho\right)^{1/2}\left(\frac{1}{4}\int_0^1\left|\int_r^1 \rho^{-1}f_2(\rho)\mathrm{d}\rho\right|^2r^{3}\mathrm{d}r\right)^{1/2}\\
   &\leq 16\left(\frac{1}{4}\|g_0\|_2^2\right)^{1/2}\left(\frac{1}{8}\frac{1}{j_1^2}\|g_2\|_2^2\right)^{1/2}\\
   &=\frac{2\sqrt{2}}{j_1}\|g_0\|_2\|g_2\|_2,
\end{align*}
and
$$8\int_0^1\left|\int_0^r \rho f_0(\rho)\mathrm{d}\rho \right|^2r^{-1}\mathrm{d}r+2\left|\int_0^1\rho f_0(\rho)\mathrm{d}\rho\right|^2 \leq3\int_0^1\rho |f_0(\rho)|^2\mathrm{d}\rho=\frac{3}{2}\|g_0\|_2^2.$$
On the other hand, it follows from (\ref{be-j2}) and (\ref{gd-def}) that
$$8\int_0^1\left|\int_r^1 \rho^{-1}f_2(\rho)\mathrm{d}\rho \right|^2r^{3}\mathrm{d}\rho\leq\frac{8}{j_1^2}\int_{0}^1\rho|f_2(\rho)|^2\mathrm{d}\rho=\frac{4}{j_1^2}\|g_2\|_2^2.$$
Then
\begin{align*}
 \|\mathcal{P}[g_0+g_2]\|_2^2& \leq\frac{3}{2}\|g_0\|_2^2+\frac{4}{j_1^2}\|g_2\|_2^2+\frac{2\sqrt{2}}{j_1}\|g_0\|_2\|g_2\|_2 \\
   & \leq \left(\frac{3}{2}+\frac{2}{j_1^2}\right)\left(\|g_0\|_2^2+\|g_2\|_2^2\right).
\end{align*}
This shows that
$$M=\max\bigg\{\sup\frac{\|\mathcal{P}[g_0+g_2]\|_2^2}{\|g_0\|_2^2+\|g_2\|_2^2}, \ \ \frac{8}{j_2^2},\ \  \frac{8}{j_1^2},\ \  \frac{4}{j_0^2},\ \ \frac{1}{3}\bigg\}\in\left(\frac{4}{j_0^2}, \frac{3}{2}+\frac{2}{j_1^2}\right).$$

\begin{description}
  \item[Step 4] {\bf Determination of }
  $$\sup\frac{\|\mathcal{P}[g_0+g_2]\|_2^2}{\|g_0\|_2^2+\|g_2\|_2^2}.$$
\end{description}
Let
$$\sigma(r)=\frac{1}{r}\int_0^r\rho|f_0(\rho)|\mathrm{d}\rho\ \ \ \mbox{and}\ \ \ \upsilon(r)=r\int_r^1\rho^{-1}|f_2(\rho)|\mathrm{d}\rho.$$
Then it follows from (\ref{july-4-eq1}) that
$$\|\mathcal{P}[g_0+g_2]\|_2^2\leq8\int_0^1r(\sigma(r)^2+\upsilon(r)^2)\mathrm{d}r+2\sigma(1)^2+16\sigma(1)\int_0^1r^2\upsilon(r)\mathrm{d}r,$$
where the equality holds if $f_0(\rho)$ and $f_2(\rho)$ both are positive real functions.

The definition of $\sigma$ and $\upsilon$ now gives the following formulas:
$$|f_0(r)|=\frac{\sigma(r)+r\sigma'(r)}{r}\ \ \ \mbox{and}\ \ \ |f_2(r)|=\frac{\upsilon(r)-r\upsilon'(r)}{r}.$$
Thus
$$\|g_0\|_2^2=2\int_0^1r|f_0(r)|^2\mathrm{d}r=2\int_0^1\frac{1}{r}\big(\sigma(r)+r\sigma'(r)\big)^2\mathrm{d}r$$
and
$$\|g_2\|_2^2=2\int_0^1r|f_2(r)|^2\mathrm{d}r=2\int_0^1\frac{1}{r}\big(\upsilon(r)-r\upsilon'(r)\big)^2\mathrm{d}r.$$
In what follows, we should find the best constant $M$ such that
\begin{align*}
\|\mathcal{P}[g_0+g_2]\|_2^2&\leq8\int_0^1r\big(\sigma(r)^2+\upsilon(r)^2\big)\mathrm{d}r+2\sigma(1)^2+16\sigma(1)\int_0^1r^2\upsilon(r)\mathrm{d}r \\
  & \leq 2M\int_0^1\frac{1}{r}\bigg(\big(\sigma(r)+r\sigma'(r)\big)^2+\big(\upsilon(r)-r\upsilon'(r)\big)^2\bigg)\mathrm{d}r,
\end{align*}
and prove that the above equalities can be attained.

For simplicity, let
$$X=8\int_0^1r\big(\sigma(r)^2+\upsilon(r)^2\big)\mathrm{d}r+2\sigma(1)^2+16\sigma(1)\int_0^1r^2\upsilon(r)\mathrm{d}r$$
and
$$Y=2\int_0^1\frac{1}{r}\bigg(\big(\sigma(r)+r\sigma'(r)\big)^2+\big(\upsilon(r)-r\upsilon'(r)\big)^2\bigg)\mathrm{d}r.$$
We now use the Lagrange multiplier method to maximize the following function
$$Z=\frac{X}{Y}.$$

According to the definition of $\sigma$ and $\upsilon$, we have the following constraints:
$$\sigma(0)=\upsilon(0)=\upsilon(1)=0$$
and (due to the definition of $f_2$)
$$\lim_{r\to0}\frac{\upsilon(r)-r\upsilon'(r)}{r}=\lim_{r\to0}\left(\frac{\upsilon(r)-r\upsilon'(r)}{r}\right)'=0.$$
Assume that $\sigma(1)=x$. Then $\sigma(1)^2=2\int_0^1\sigma(r)\sigma'(r)dr$. We then rewrite $Y$ as follows:
$$Y=2\int_0^1  \bigg(\frac{\sigma(r)^2+\upsilon(r)^2}{r}+r(\sigma'(r)^2+\upsilon'(r)^2)-2\upsilon(r)\upsilon'(r)+x^2\bigg) dr.$$
Now, the corresponding Lagrange function is
\begin{align*}
  L & =8r(\sigma(r)^2 +\upsilon(r)^2)+2x^2 +16 x r^2 \upsilon(r) \\
  & \ \ \ -2\lambda \bigg(\frac{\sigma(r)^2+\upsilon(r)^2}{r}+r(\sigma'(r)^2+\upsilon'(r)^2)-2\upsilon(r)\upsilon'(r)+x^2\bigg).
\end{align*}
The stationary points of $L$ are solutions to the following system:
$\frac{\partial}{\partial_r}\left(\frac{d L}{d\sigma'}\right)=\frac{d L}{d\sigma}$ and $\frac{\partial}{\partial_r}\left(\frac{d L}{d\upsilon'}\right)=\frac{d L}{d\upsilon}$,
i.e.,
$$\left\{
\begin{array}
{r@{\ }l}
&-\frac{\lambda \sigma}{r}+4 r \sigma+\lambda \sigma'+\lambda r \sigma''=0\\
\\
&4x r^2-\frac{\lambda \upsilon}{r}+4r\upsilon+\lambda\upsilon'+\lambda r \upsilon''=0.
\end{array}\right.$$

The solutions to the first equation are
$$\sigma(r)=c_1J_1\left(\frac{2r}{\sqrt{\lambda}}\right)+c_2\mathcal{N}_1\left(\frac{2r}{\sqrt{\lambda}}\right),$$
where $J_1$ is the Bessel function and $\mathcal{N}_1$ is the Hankel's function.
Since $\sigma(0)=0$, we then have $c_2=0$. Thus $\sigma(r)=c_1J_1(2r/\sqrt{\lambda})$.

The second equation has the following solutions
$$\upsilon(r)=-xr+c_3J_1\left(\frac{2r}{\sqrt{\lambda}}\right)+c_4\mathcal{N}_1\left(\frac{2r}{\sqrt{\lambda}}\right).$$
Due to the constraints $\upsilon(0)=\upsilon(1)=0$ and the assumption $\sigma(1)=x$, we obtain
\be\label{extremal-func}\sigma(r)=x\frac{J_1\left(\frac{2r}{\sqrt{\lambda}}\right)}{J_1\left(\frac{2}{\sqrt{\lambda}}\right)}\ \ \ \mbox{and}\ \ \ \upsilon(r)=-xr+\frac{xJ_1\left(\frac{2r}{\sqrt{\lambda}}\right)}{J_1\left(\frac{2}{\sqrt{\lambda}}\right)}.\ee

By inserting the above functions $\sigma$ and $\upsilon$ into the functions $X$ and $Y$, we get
$$X=8x^2\left(1+\frac{J_0\left(\frac{2}{\sqrt{\lambda}}\right)^2}{J_1\left(\frac{2}{\sqrt{\lambda}}\right)^2}-\frac{\sqrt{\lambda}J_0\left(\frac{2}{\sqrt{\lambda}}\right)}{J_1\left(\frac{2}{\sqrt{\lambda}}\right)}\right)$$
and
$$Y= -4 x^2+\frac{8 x^2}{\lambda}+\frac{8 x^2 J_0\left(\frac{2}{\sqrt{\lambda}}\right)^2}{\lambda J_1\left(\frac{2}{\sqrt{\lambda}}\right)^2}.$$
Then
$$Z=\frac{X}{Y}=\frac{2 \lambda^2 J_0\left(\frac{2}{\sqrt{\lambda}}\right)^2-2 \lambda^{5/2} J_0\left(\frac{2}{\sqrt{\lambda}}\right) J_1\left(\frac{2}{\sqrt{\lambda}}\right)+2 \lambda^2 J_1\left(\frac{2}{\sqrt{\lambda}}\right)^2}{2 \lambda J_0\left(\frac{2}{\sqrt{\lambda}}\right)^2+(2-\lambda) \lambda J_1\left(\frac{2}{\sqrt{\lambda}}\right)^2}.$$
According to Lemma \ref{step-4-lem6-7-9}, to find the maximum value of $Z$, we only need to solve the equation $Z=\lambda$, or equivalently,
\be\label{july-7-eq1}\lambda^{3/2} J_1\left(\frac{2}{\sqrt{\lambda}}\right) \left(-2 J_0\left(\frac{2}{\sqrt{\lambda}}\right)+\sqrt{\lambda} J_1\left(\frac{2}{\sqrt{\lambda}}\right)\right)=0.\ee

\begin{description}
  \item[Claim 1.1] The equation (\ref{july-7-eq1}) has only one solution at the interval $\left(4/j_0^2,\  3/2+2/j_1^2\right)$, and the unique solution is $\lambda_0\approx 1.180.$
\end{description}
This can be proved as follows:
Let $\delta=2/\sqrt{\lambda}$, we then reformulate (\ref{july-7-eq1}) as follows
$$H(\delta) :=J_1(\delta)-\delta J_0(\delta)=0,$$
where $\delta\in(2/\sqrt{3/2+2/j_1^2},\ j_0)$.
Since $H(2/\sqrt{3/2+2/j_1^2})<0$, $H(j_0)>0$, and $H'(\delta) =\left(\delta^2-1\right) J_1(\delta)/\delta>0$,  we see that $H(\delta)=0$ has only one solution.
Solving the equation $H(\delta)=0$, we have $\delta\approx1.841$, which shows that Claim 1.1 holds true.

Based on the above observations of Step 2, Step 3, and  Step 4, we have found the best constant $M=\lambda_0$ such that
$\|\mathcal{P}[\varphi]\|_2^2\leq M\|\varphi\|_2^2$, for any $\varphi\in L^2(\ID, \mathrm{d}A)$.
The equality holds, if we choose $\varphi_0(z)=f_0(r)+f_2(r)e^{2it}$,
where
$$f_0(r)=\frac{2J_0\left(\frac{2r}{\sqrt{\lambda_0}}\right)}{\sqrt{\lambda_0}J_1\left(\frac{2}{\sqrt{\lambda_0}}\right)} \ \ \ \mbox{and}\ \ \ f_2(r)=\frac{2J_2\left(\frac{2r}{\sqrt{\lambda_0}}\right)}{\sqrt{\lambda_0}J_1\left(\frac{2}{\sqrt{\lambda_0}}\right)}.$$

Hence
$$\|\mathcal{P}\|_2=\alpha=\left(\sup\frac{\|\mathcal{P}[g_0+g_2]\|_2^2}{\|g_0\|_2^2+\|g_2\|_2^2}\right)^{1/2}=\sqrt{\lambda_0},$$
which completes the proof.
\qed

\subsection*{Proof of Theorem \ref{infinity-norm}}
Suppose $\varphi\in L^{p}(\ID, \mathrm{d}A)$, where $p>2$, and let $z=re^{it}\in\ID$.
It follows from H\"older's inequality for integral that
\begin{align*}
  |\mathcal{P}[\varphi](z)|& =\left|\int_{\ID}\left(\frac{\varphi(w)}{z-w}+\frac{z}{1-\bar{w}z}\overline{\varphi(w)}\right)\mathrm{d}A(w)\right| \\
   & \leq\left[ \left(\int_{\mathbb{D}}   \frac{\mathrm{d} A(w)}{|z-w|^q} \right)^{\frac{1}{q}}+ \left(\int_{\mathbb{D}} \frac{|z|^q}{|1-\bar{w}z|^q} \mathrm{d} A(w)\right)^{\frac{1}{q}}\right]\|\varphi\|_{p}\\
   &\leq2^{1-\frac{1}{q}} \left[\int_{\mathbb{D}}   \left(\frac{1}{|z-w|^q}+\frac{|z|^q}{|1-\bar{w}z|^q}\right) \mathrm{d} A(w)\right]^{\frac{1}{q}} \|\varphi\|_{p},
\end{align*}
where $q=p/(p-1)\in[1, 2)$.
Therefore,
$$\|\mathcal{P}\|_{L^p\to L^\infty}\leq2^{1-\frac{1}{q}}\sup_{z \in \mathbb{D}} \left[\int_{\mathbb{D}}   \left(\frac{1}{|z-w|^q}+\frac{|z|^q}{|1-\bar{w}z|^q}\right) \mathrm{d} A(w)\right]^{\frac{1}{q}}.$$
In what follows, we should find such a supremum and prove that equality holds.

We start with the first integral. By using $w=\frac{z-\tau}{1-\bar{z}\tau}$, we then rewrite it as follows:
\begin{align*}
\int_{\mathbb{D}}   \frac{1}{|z-w|^q} \mathrm{d} A(w)& =\int_{\mathbb{D}} \frac{1}{|z-\frac{z-\tau}{1-\bar{z}\tau}|^q}\frac{(1-|z|^2)^2}{|1-z\bar{\tau}|^4}  \mathrm{d} A(w)  \\
   & =(1-|z|^2)^{2-q}   \int_{\mathbb{D}}   \frac{1}{|\tau|^q|1-z\bar{\tau}|^{4-q}}\mathrm{d}A(w).
\end{align*}
The following formula can be proved by using Taylor expansion and Parseval's formula:
\be\label{parseval}\frac{1}{2\pi}\int_0^{2\pi}\frac{\mathrm{d}\theta}{|1-ze^{i\theta}|^{2\beta}}=\sum_{n=0}^\infty\left(\frac{\Gamma(n+\beta)}{n!\Gamma(\beta)}\right)^2|z|^{2n},\ee
where $\beta>0$, $z\in\mathbb{D}$, and $\Gamma$ is the Gamma function.
Since $\tau\in\ID$, by using hypergeometric functions given by (\ref{bz-defn-2.1}) together with the above equation (\ref{parseval}), we have
\begin{align*}
   \int_{\mathbb{D}}   \frac{1}{|\tau|^q|1-z\bar{\tau}|^{4-q}}\mathrm{d}A(w) &=2\sum_{n=0}^{\infty}\left(\frac{\Gamma(n+2-q/2)}{n!\Gamma(2-q/2)}\right)^2\frac{|z|^{2n}}{2n+2-q}\\
   &=\frac{2}{2-q}\ _2 F_1\left(1-\frac{q}{2},2-\frac{q}{2};1;|z|^2\right).
\end{align*}
Now, applying Euler's formula, the first integral can be given as follows:
$$\int_{\mathbb{D}}   \frac{1}{|z-w|^q} \mathrm{d} A(w)=\frac{2}{2-q}\ _2 F_1\left(\frac{q}{2},\frac{q}{2}-1;1;|z|^2\right).$$
On the other hand, by using (\ref{parseval}) again, the second integral can be given as follows:
$$\int_{\mathbb{D}}\frac{|z|^q}{|1-\bar{w}z|^q} d A(w)=\ _2{F}_1  \left(\frac{q}{2},\frac{q}{2};2;|z|^2\right)|z|^q.$$	

\begin{description}
  \item[Claim 2.1] The function
  $$\Phi(t)=\frac{2}{2-q}\ _2 F_1\left(\frac{q}{2},\frac{q}{2}-1;1;t\right)+\ _2{F}_1  \left(\frac{q}{2},\frac{q}{2};2;t\right)t^{\frac{q}{2}}$$
  is strictly increasing on $[0, 1]$.
\end{description}
This can be proved as follows: Elementary calculations show that
$$ \frac{2}{q}\Phi'(t)= \frac{q}{4}t^{\frac{q}{2}}\ _2 F_1\left(\frac{q}{2}+1,\frac{q}{2}+1;3;t\right)+t^{\frac{q}{2}-1}\ _2 F_1\left(\frac{q}{2},\frac{q}{2};2;t\right )-\ _2 F_1  \left(\frac{q}{2}+1,\frac{q}{2};2;t\right).$$
By using the definition of hypergeometric function, we then arrive
\begin{align*}
  \frac{2}{q}\Phi'(t) &=t^{\frac{q}{2}-1}+t^{\frac{q}{2}-1} \sum_{k=1}^{+\infty}\bigg(\frac{2}{q}\frac{k(\frac{q}{2})_k(\frac{q}{2})_k}{(2)_k k!}+\frac{(\frac{q}{2})_{k}(\frac{q}{2})_{k}}{(2)_{k} k!}\ \bigg) t^{k}- \sum_{k=0}^{+\infty}\frac{(\frac{q}{2}+1)_k(\frac{q}{2})_k}{(2)_k k!}t^{k}\\
  &=t^{\frac{q}{2}-1}\sum_{k=0}^{+\infty}\bigg(1+\frac{2k}{q}\bigg)\frac{(\frac{q}{2})_k(\frac{q}{2})_k}{(2)_k k!} t^{k}- \sum_{k=0}^{+\infty}\frac{(\frac{q}{2}+1)_k(\frac{q}{2})_k}{(2)_k k!}t^{k}.
\end{align*}
The assumption $p>2$ now implies that $q/2-1<0$, and thus $t^{q/2-1}\geq1$, since $t=|z|^2\in [0, 1]$. Then
$$\frac{2}{q}\Phi'(t) \geq\sum_{k=0}^{+\infty}\bigg(1+\frac{2k}{q}\bigg)\frac{(\frac{q}{2})_k(\frac{q}{2})_k}{(2)_k k!} t^{k}- \sum_{k=0}^{+\infty}\frac{(\frac{q}{2}+1)_k(\frac{q}{2})_k}{(2)_k k!}t^{k}=0,$$
where the last equality holds because for all non-negative integers $k$ we have:
$$\bigg(1+\frac{2k}{q}\bigg)\frac{(\frac{q}{2})_k(\frac{q}{2})_k}{(2)_k k!}=\frac{(\frac{q}{2}+1)_k(\frac{q}{2})_k}{(2)_k k!}.$$
This shows that Claim 2.1 is true.

By appealing to Gauss' theorem, we have
$$\Phi(1)=\frac{2\Gamma(2-q)}{\Gamma^2(2-\frac{q}{2})}.$$
Therefore,
$$\int_{\mathbb{D}}   \left(\frac{1}{|z-w|^q}+\frac{|z|^q}{|1-\bar{w}z|^q}\right) \mathrm{d} A(w)\leq\max_{r\in[0, 1]}\Phi(r)=\frac{2\Gamma(2-q)}{\Gamma^2(2-\frac{q}{2})},$$
which implies that
$$\|\mathcal{P}\|_{L^p\to L^{\infty}}\leq2\left(\frac{\Gamma(2-q)}{\Gamma^2(2-\frac{q}{2})}\right)^{\frac{1}{q}}.$$

To show the above equality holds, let $\varphi_0(w)=\frac{i(1-w)}{|1-w|^{1+\frac{q}{p}}}$.
Then
$$\|\varphi_0\|_p=\left(\int_{\ID}\frac{1}{|1-w|^q}\mathrm{d}A(w)\right)^{\frac{1}{p}}=\left(\frac{\Gamma(2-q)}{\Gamma^2(2-\frac{q}{2})}\right)^{\frac{1}{p}}.$$
Note that for $z=1$, the mean-inequality becomes equality, we then have
$$\max_{z\in\ID}\frac{|\mathcal{P}[\varphi_0](z)|}{\|\varphi_0\|_p}=\frac{|\mathcal{P}[\varphi_0](1)|}{\|\varphi_0\|_p}=2\left(\frac{\Gamma(2-q)}{\Gamma^2(2-\frac{q}{2})}\right)^{\frac{1}{q}}.$$
This shows that
$$\|\mathcal{P}\|_{L^p\to L^{\infty}}\geq2\left(\frac{\Gamma(2-q)}{\Gamma^2(2-\frac{q}{2})}\right)^{\frac{1}{q}}.$$
Therefore, we conclude that
$$\|\mathcal{P}\|_{L^p\to L^{\infty}}=2{\bigg(\frac{\Gamma(2-q)}{\Gamma^2(2-\frac{q}{2})}\bigg)}^{\frac{1}{q}}.$$

If in particular $p=\infty$, i.e., $q=1$, then following the above proof, we have
$$\|\mathcal{P}\|_{\infty}\leq\frac{8}{\pi}.$$
The extremal function $\varphi_0(w)=\frac{i(1-w)}{|1-w|}$,  and limiting proces when $z \rightarrow 1$, show that the above equality can be attained, which completes the proof.
\qed

\begin{rem}
If $1\leq p\leq2$, then $\mathcal{P}$ will not send $L^p(\ID, \mathrm{d}A)$ to $L^{\infty}(\ID, \mathrm{d}A)$.
We only need to consider the case of $p=2$, because $L^2(\ID, \mathrm{d}A)\subseteq L^p(\ID, \mathrm{d}A)$.

Let
$$\varphi_0(w)=\frac{w}{|w|^2\log\frac{2}{|w|}}.$$
Then
$$\|\varphi_0\|_2^2=2\int_0^1\frac{1}{r\log^2\frac{2}{r}}\mathrm{d}r=\frac{2}{\log2},$$
which shows that $\varphi_0\in L^2(\ID, \mathrm{d}A)$. However,
$$|\mathcal{P}[\varphi_0](0)|=2\int_0^1\frac{1}{r\log\frac{2}{r}}\mathrm{d}r=\infty,$$
which shows that $\mathcal{P}[\varphi_0]\notin L^{\infty}(\ID, \mathrm{d}A)$.

\end{rem}
\section{Some calculations related to $\|\mathcal{P}\|_1$}\label{L1norm}
Suppose $\varphi\in L^1(\ID, \mathrm{d}A)$ and let $\beta=\arg\varphi(w)$. Then
\begin{align*}
  \|\mathcal{P}[\varphi]\|_1 & =\int_{\ID}\left|\int_{\ID}\frac{\varphi(w)}{w-z}+\frac{z\overline{\varphi(w)}}{1-\bar{w}z}\mathrm{d}A(w)\right|\mathrm{d}A(z) \\
  &\leq\int_{\ID}\left(\int_{\ID}\left|\frac{e^{i\beta}}{w-z}+\frac{ze^{-i\beta}}{1-\bar{w}z}\right||\varphi(w)|\mathrm{d}A(w)\right)\mathrm{d}A(z)\\
  &\leq\int_{\ID}|\varphi(w)|\mathrm{d}A(w)\sup_{w\in\ID}\int_{\ID}\left|\frac{1}{we^{-i\beta}-ze^{-i\beta}}+\frac{ze^{-i\beta}}{1-\overline{we^{-i\beta}}ze^{-i\beta}}\right|\mathrm{d}A(z).
\end{align*}
By letting $\eta=ze^{-i\beta}$ and $\omega=we^{-i\beta}\in\ID$, we have
$$ \|\mathcal{P}[\varphi]\|_1\leq\sup_{\omega\in\ID}\int_{\ID}\left|\frac{1}{\omega-\eta}+\frac{\eta}{1-\bar{\omega}\eta}\right|\mathrm{d}A(\eta)\|\varphi\|_1.$$
Therefore,
$$\|\mathcal{P}\|_1 \leq \sup_{w\in\ID}\int_{\ID}\left|\frac{1}{w-z}+\frac{z}{1-\bar{w}z}\right|\mathrm{d}A(z).$$
Fix $w_0\in \ID$ and let $\varphi_{w_0}(z)=\frac{w_0-z}{1-\bar{w}_0z}$ be the M\"obius transform.
Suppose $\chi_{\Omega}$ is the characteristic function of $\Omega$, i.e.,
$$\chi_{\Omega}(w)=\left\{
\begin{array}
{r@{\ }l}
1, \ \ \ \ \ \   & \mbox{if}\ \ \ w\in\Omega\\
\\
0, \ \ \ \ \ \ &  \mbox{if}\ \ \  w\in\mathbb{C}\backslash\Omega.
\end{array}\right.$$
Consider the following function:
$$\varphi_n(z)=n^2\chi_{\frac{1}{n}\ID}(\varphi_{w_0}(z)).$$ Then
$$\|\varphi_n\|_1=\int_{\ID}\frac{(1-|w_0|^2)^2}{|1-\bar{w}_0\frac{1}{n}\xi|^4}\mathrm{d}A(\xi)\to(1-|w_0|^2)^2\ \ \ \mbox{as}\ \ \ n\to \infty.$$
On the other hand, we have
$$\frac{\|\mathcal{P}[\varphi_n]\|_1}{\|\varphi_n\|_1}\to\int_{\ID}\left|\frac{1}{w_0-z}+\frac{z}{1-\bar{w}_0z}\right|\mathrm{d}A(z)\ \ \ \mbox{as}\ \ \ n\to \infty.$$
This shows that
\begin{align*}
  \|\mathcal{P}\|_1 & = \sup_{w\in\ID}\int_{\ID}\left|\frac{1}{w-z}+\frac{z}{1-\bar{w}z}\right|\mathrm{d}A(z) \\
  & =\sup_{w\in\ID}(1-|w|^2)\int_{\ID}\frac{|w-\bar{w}-\zeta+\frac{1}{\zeta}|}{|1-\bar{w}\zeta|^4}\mathrm{d}A(\zeta),
\end{align*}
where the last equality holds, after change of variables $\zeta=\frac{w-z}{1-\bar{w}z}$.

\begin{rem}
It seems to the authors that the above supremum can be attained at the point $w=0$, i.e.,
$$\|\mathcal{P}\|_1=\int_{\ID}\left|z-\frac{1}{z}\right|\mathrm{d}A(z).$$
Unfortunately, we fail to prove this.
We only give some calculations connected to the conjectured value of $\|\mathcal{P}\|_1$ as follows:
Suppose $z=re^{it}\in\ID$. Then
\begin{align*}
  \int_{\ID}\left|z-\frac{1}{z}\right|\mathrm{d}A(z) &=\frac{1}{\pi}\int_0^1\mathrm{d}r\int_0^{2\pi}\sqrt{1+r^4-2r^2\cos(2t)}\mathrm{d}t\\
   & =\frac{1}{\pi}\int_0^1(1-r^2)\mathrm{d}r\int_0^{2\pi}\sqrt{1+\left(\frac{2r}{1-r^2}\right)^2\sin^2(t)}\mathrm{d}t.
\end{align*}
Recall that elliptic integral of the second kind is defined as follows:
\begin{align*}
  E(k^2) & =\int_0^{\frac{\pi}{2}}\sqrt{1-k^2\sin^2(t)}\mathrm{d}t \\
   & =\frac{\pi}{2}\left(1-\sum_{n\geq1}\left(\frac{(2n-1)!!}{(2n)!!}\right)^2\frac{k^{2n}}{2n-1}\right),
\end{align*}
where $|k|\leq1$.
Then
$$\int_0^{2\pi}\sqrt{1-k^2\sin^2(t)}\mathrm{d}t=2E(k^2)+2\sqrt{1-k^2}E\left(\frac{-k^2}{1-k^2}\right).$$
Therefore,
\begin{align*}
  \int_{\ID}\left|z-\frac{1}{z}\right|\mathrm{d}A(z) &=\frac{2}{\pi}\int_0^1\left\{(1-r^2)E\left(\frac{-4r^2}{(1-r^2)^2}\right)+(1+r^2)E\left(\frac{4r^2}{(1+r^2)^2}\right)\right\}\mathrm{d}r\\
   &\approx2.10441 (\mbox{by using the mathematica software}).
\end{align*}
\end{rem}
\section{Proof of Theorem \ref{thm-hilbert-1}}\label{sec-4}
We first prove two lemmas as follows:

\begin{lem}\label{lem-Feb-5-20}
Let $p(w)=a_{m, n}\bar{w}^nw^{m}\chi_{\ID}(w)$, where $m$, $n$ are non-negative integers and $a_{m, n}$ are complex numbers. Then
for $z\in\ID$, we have
\be\label{524-2}\mathcal{H}[p](z)=\left\{
\begin{array}
{r@{\ }l}
\frac{m}{n+1}a_{m, n}z^{m-1}\bar{z}^{n+1}-\frac{m-n-1}{n+1}a_{m, n}z^{m-n-2},  \ \ \ \ \ \    & \emph{if}\ \ \ m-n>1\\
\\
\frac{m}{n+1}a_{m, n}z^{m-1}\bar{z}^{n+1}-\frac{1-m+n}{n+1}\bar{a}_{m, n}z^{-m+n}, \ \ \ \ \ \ &  \emph{if}\ \ \ m-n\leq 1.
\end{array}\right.\ee
\end{lem}
\bpf
Fix $z\in\ID$, according to {\it Green's formula} in the form $$-\int_{P(a,r,R)}\frac{\partial f(w)}{\partial \bar w}dw\wedge d\bar{ w}=\int_{\partial P(a,r,R)}f(w)dw,$$ where $P(a,r,R)=\{w\in\mathbb{C}: |w|<R, |w-a|>r\}$, $|a|+r<R$, and $dw\wedge d\bar{w}=-2\pi i\mathrm{d}A(w)$, we have
\begin{align*}
  \mathcal{S}[p](z) & =-a_{m, n}\lim_{\varepsilon\to0}\int_{\ID\setminus\ID(z, \varepsilon)}\frac{\bar{w}^nw^{m}}{(w-z)^2}\mathrm{d}A(w) \\
  & =\frac{-a_{m, n}}{2\pi i}\lim_{\varepsilon\to0}\left(\int_{\mathbb{T}}\frac{\bar{\zeta}^{n+1}\zeta^{m}\mathrm{d}\zeta}{(n+1)(\zeta-z)^2}-\int_{\partial\ID(z, \varepsilon)}\frac{\bar{\zeta}^{n+1}\zeta^{m}\mathrm{d}\zeta}{(n+1)(\zeta-z)^2}\right).
\end{align*}
Applying Cauchy residue theorem to the first integral, we have
$$\frac{-a_{m, n}}{2\pi i}\int_{\mathbb{T}}\frac{\bar{\zeta}^{n+1}\zeta^{m}\mathrm{d}\zeta}{(n+1)(\zeta-z)^2}=\left\{
\begin{array}
{r@{\ }l}
-\frac{m-n-1}{n+1}a_{m, n}z^{m-n-2}, \ \    &\mbox{if}\ \ \   m-n>1\\
\\
0, \ \  & \mbox{if}\ \ \ m-n\leq 1.
\end{array}\right.$$
By taking $\zeta=z+\varepsilon\xi$, and again, applying Cauchy residue theorem to the second integral, we get
\begin{align*}
  \frac{a_{m, n}}{2\pi i}\lim_{\varepsilon\to0}\int_{\partial\ID(z, \varepsilon)}\frac{\bar{\zeta}^{n+1}\zeta^{m}\mathrm{d}\zeta}{(n+1)(\zeta-z)^2} &=\frac{a_{m, n}}{2\pi i(n+1)}\lim_{\varepsilon\to0}\int_{\mathbb{T}}\frac{(\bar{z}+\varepsilon\bar{\xi})^{n+1}(z+\varepsilon\xi)^{m}}{\varepsilon\xi^{2}}\mathrm{d}\xi \\
   & =\frac{m}{n+1}a_{m, n}z^{m-1}\bar{z}^{n+1}.
\end{align*}
Therefore,
\be\label{beurling-ope}\mathcal{S}[p](z)=\left\{
\begin{array}
{r@{\ }l}
\frac{m}{n+1}a_{m, n}z^{m-1}\bar{z}^{n+1}-\frac{m-n-1}{n+1}a_{m, n}z^{m-n-2}, \ \ \ \ \ \ \    & \mbox{if}\ \ \ m-n>1\\
\\
\frac{m}{n+1}a_{m, n}z^{m-1}\bar{z}^{n+1}, \ \ \ \ \ \ & \mbox{if}\ \ \ m-n\leq 1.
\end{array}\right.\ee

On the other hand, direct calculations show that
\begin{align*}
  \mathfrak{B}[\bar{p}](z)&=\bar{a}_{m, n}\sum_{k=0}^{\infty}(k+1)z^k\int_{\ID}w^{n}\bar{w}^{m+k}\mathrm{d}A(w)\\
   &=\frac{1-m+n}{n+1}\bar{a}_{m, n}z^{-m+n},
\end{align*}
where $m-n\leq0$. Therefore,
$$\mathfrak{B}[\bar{p}](z)=\left\{
\begin{array}
{r@{\ }l}
0, \ \ \ \ \ \ \    & \mbox{if}\ \ \ m-n>1\\
\\
\frac{1-m+n}{n+1}\bar{a}_{m, n}z^{-m+n}, \ \ \ \ \ \ &  \mbox{if}\ \ \ m-n\leq 1.
\end{array}\right.$$

Based on the above discussions, we see that (\ref{524-2}) holds true since
$$\mathcal{H}[p](z)=\mathcal{S}[p](z)-\mathfrak{B}[\bar{p}](z).$$
\epf

\begin{lem}\label{orth-lemma}
Let $g_d$ be given by $(\ref{6-28-gd})$.
Then for any two different integers $d_1$, $d_2$, the functions $\mathcal{H}[g_{d_1}](z)$ and $\mathcal{H}[g_{d_2}](z)$ are orthogonal in $\ID$, i.e.,
\be\label{orthogonal}\int_{\ID}\mathcal{H}[g_{d_1}](z)\overline{\mathcal{H}[g_{d_2}](z)}\mathrm{d}A(z)=0\ \ \ \ \mbox{for any}\ \ \   d_1\neq d_2.\ee
\end{lem}
\begin{proof}
Fix $z=re^{it}\in\ID$. Sine $\mathcal{H}[g_d](z)=\sum_{n\geq\max\{0, -d\}}\mathcal{H}[b_nw^{n+d}\bar{w}^n]$,
it follows from Lemma \ref{lem-Feb-5-20} that
\be\label{hgd}\mathcal{H}[g_d](z)=\left\{
\begin{array}
{r@{\ }l}
\sum\limits_{n=0}^{\infty}b_n\left(\frac{n+d}{n+1}z^{n+d-1}\bar{z}^{n+1}-\frac{d-1}{n+1}z^{d-2}\right), \ \ \ \ \ \ \    & \mbox{if}\ \ \  d>1
\\
\sum\limits_{n=0}^{\infty}b_nz^{n}\bar{z}^{n+1}, \ \ \ \ \ \ &  \mbox{if}\ \ \  d=1
\\
\sum\limits_{n=-d}^{\infty}\left(\frac{n+d}{n+1}b_nz^{n+d-1}\bar{z}^{n+1}-\bar{b}_n\frac{1-d}{n+1}z^{-d}\right), \ \ \ \ \ \ &  \mbox{if}\ \ \  d< 1.
\end{array}\right.\ee
Now suppose $d_1$ and $d_2$ are two different integers.
Long but straightforward calculations together with the following fact
$$\int_0^{2\pi}e^{ikt}dt=0,\ \ \ \mbox{for any}\ \ \ k\neq0,$$
show that for the cases: $d_1$, $d_2>1$, or $d_1$, $d_2<1$, or $d_1=1$, $d_2\neq1$, the formula (\ref{orthogonal}) always holds true.
We only prove the following case.
\begin{description}
  \item[Case 1] $d_1>1$ and $d_2< 1$.
\end{description}
Elementary calculations show that
\begin{align*}
  \mathcal{H}&[g_{d_1}](z)\overline{\mathcal{H}[g_{d_2}](z)}=\\
   &\sum_{n\geq0,\, l\geq-d_2}\frac{(l+d_2)b_n\bar{b}_le^{i(d_1-d_2)t}}{(n+1)(l+1)}\left\{(n+d_1)r^{2n+2l+d_1+d_2}-(d_1-1)r^{2l+d_1+d_2-2}\right\} \\
   &+\sum_{n\geq0,\, l\geq-d_2} \frac{(1-d_2)b_n{b}_le^{i(d_1+d_2-2)t}}{(n+1)(l+1)}\{(d_1-1)r^{d_1-d_2-2}-(n+d_1)r^{2n+d_1-d_2}\}.
\end{align*}

If $d_1+d_2-2\neq0$, then it is easy to see that (\ref{orthogonal}) holds true.

If $d_1+d_2-2=0$, then
$$  \int_{\ID}\mathcal{H}[g_{d_1}](z)\overline{\mathcal{H}[g_{d_2}](z)}\mathrm{d}A(z) =2\sum_{n\geq0,\, l\geq-d_2}b_n{b}_l\frac{(1-d_2)(d_1+d_2-2)}{(l+1)(d_1-d_2)(2n+d_1-d_2+2)}=0,$$
which again shows that (\ref{orthogonal}) holds true.
Therefore, the proof is complete.
\end{proof}

\subsection*{Proof of Theorem \ref{thm-hilbert-1}}
To prove Theorem \ref{thm-hilbert-1}, following the proof of Theorem \ref{general-2norm}, it suffices to show that
\be\label{June-27-1}\|\mathcal{H}[\varphi]\|_{2}= \|\varphi\|_{2}\ee
for any functions $\varphi(w)=\sum_{d=-\infty}^{\infty}g_d(w)\in L^2(\ID, \mathrm{d}A)$,
where $g_d(w)$ is given by (\ref{6-28-gd}).

On the other hand, we already show by Lemma \ref{orth-lemma} that $\mathcal{H}[g_{d_1}]$ and $\mathcal{H}[g_{d_2}]$ are orthogonal functions in the Hilbert space $L^2(\ID, \mathrm{d}A)$. Thus, $\|\mathcal{H}[\varphi]\|_2^2=\sum_{d=-\infty}^{\infty}\|\mathcal{H}[g_d]\|_{2}^2$. It is easy to see that
$\|\varphi\|_2^2=\sum_{d=-\infty}^{\infty}\|g_d\|_{2}^2$.
Now, following the proof of the corresponding results in \cite[Theorem 5.2]{kalaj1} (see also \cite[Page 180]{Anderson}), to prove (\ref{June-27-1}), we only need to show
\be\label{main-eq-710}\|\mathcal{H}[g_d]\|_2^2=\|g_d\|_2^2\ \ \ \mbox{for}\ \ \ d=0, \pm1, \pm2, \cdots .\ee
In fact, we can choose the function $\varphi(w)=g_d(w)\in L^2(\ID, \mathrm{d}A)$, for fixed $d\in \mathbb{Z}$.

We now prove (\ref{main-eq-710}) as follows:
Elementary calculations show that
$$\|g_d\|_{2}^2=\sum\limits_{n,\, l\geq\max\{0, -d\} }\frac{b_n\bar{b}_l}{n+l+d+1}.$$

Let $z=re^{it}\in\ID$. We now calculate $\|\mathcal{H}[g_d]\|_2^2$, which is given by  (\ref{hgd}), as follows:
If $d>1$, then long but straightforward calculations show that
\begin{align*}
  &\|\mathcal{H}[g_d]\|_2^2  =\int_{\ID}\left|\sum\limits_{n=0}^{\infty}b_n\left(\frac{n+d}{n+1}r^{2n+d}e^{i(d-2)t}-\frac{d-1}{n+1}r^{d-2}e^{i(d-2)t}\right)\right|^2\mathrm{d}A(z) \\
   &=2\sum\limits_{n,\,  l\geq0}\frac{b_n\bar{b}_l}{(n+1)(l+1)}\bigg\{\frac{(n+d)(l+d)}{2n+2d+2l+2}-\frac{d-1}{2}\bigg\}\\
   & =\sum\limits_{n,\,  l\geq0}\frac{b_n\bar{b}_l}{n+l+d+1}.
\end{align*}
This shows that in this case, we have $\|\mathcal{H}[g_d]\|_2^2=\|[g_d]\|_2^2$.

If $d=1$, then
$$ \|\mathcal{H}[g_1]\|_2^2=\sum\limits_{n,\,  l\geq0}b_n\overline{b_l}\frac{1}{n+l+2}=\|g_1\|_{2}^2.$$

If $d<1$, then elementary calculations show that
\begin{align*}
  &\|\mathcal{H}[g_d]\|_2^2  =\int_{\ID}\left|\sum\limits_{n=-d}^{\infty}\left(\frac{n+d}{n+1}b_nr^{2n+d}e^{i(d-2)t}-\frac{1-d}{n+1}\bar{b}_nr^{-d}e^{-idt}\right)\right|^2\mathrm{d}A(z) \\
  &=\sum\limits_{n,\,l  \geq-d}\int_{\ID}\left\{\frac{(n+d)(l+d)}{(n+1)(l+1)}b_n\bar{b}_lr^{2n+2l+2d}-\frac{(n+d)(1-d)}{(n+1)(l+1)}b_nb_lr^{2n}e^{i(2d-2)t}\right.\\
  &\ \ \ \left.-\frac{(1-d)(l+d)}{(n+1)(l+1)}\bar{b}_n\bar{b}_lr^{2l}e^{i(2-2d)t}+\frac{(1-d)^2}{(n+1)(l+1)}\bar{b}_nb_lr^{-2d}\right\}\mathrm{d}A(z).
\end{align*}
Thus
$$ \|\mathcal{H}[g_d]\|_2^2=2\sum\limits_{n,\, l\geq-d}\left(\frac{b_n\bar{b}_l(n+d)(l+d)}{(n+1)(l+1)(2n+2l+2d+2)}+\frac{\bar{b}_nb_l(1-d)}{2(n+1)(l+1)}\right).$$
Note that
$$\sum\limits_{n,\, l\geq-d}\frac{\bar{b}_nb_l(1-d)}{2(n+1)(l+1)}=\sum\limits_{n,\, l\geq-d}\frac{\bar{b}_lb_n(1-d)}{2(l+1)(n+1)}.$$
We then have
\begin{align*}
   \|\mathcal{H}[g_d]\|_2^2 & =2\sum\limits_{n,\, l\geq-d}\frac{b_n\bar{b}_l}{(n+1)(l+1)}\left(\frac{(n+d)(l+d)}{2n+2l+2d+2}+\frac{1-d}{2}\right) \\
   & =\sum\limits_{n,\,  l\geq0}\frac{b_n\bar{b}_l}{n+l+d+1}.
\end{align*}
Again, in this case, we have $\|\mathcal{H}[g_d]\|_2^2 =\|g_d\|_2^2$.

Based on the above discussions, we see that
$$\|\mathcal{H}[\varphi]\|_{2}^2\equiv\|\varphi\|_2^2,$$
for any $\varphi\in L^2(\ID, \mathrm{d}A)$.
This shows that $\mathcal{H}$ acts as an isometry in $L^2(\ID, \mathrm{d}A)$, and thus, the proof is complete.
\qed

\vspace*{5mm}
\noindent {\bf Acknowledgments}.
The research of the authors were supported by NSFs of China (No. 11501220, 11971182), NSFs of Fujian Province (No. 2016J01020, 2019J0101)
and the Promotion Program for Young and Middle-aged Teachers in Science and Technology Research of Huaqiao University (ZQN-PY402).

\end{document}